\documentclass[a4paper,11pt]{article}
\usepackage{graphicx}
\usepackage{amsmath,amsthm,amssymb,enumerate}
\usepackage{euscript,mathrsfs,amsfonts,dsfont}
\usepackage{xcolor}
\usepackage{empheq}
\usepackage{geometry}
\usepackage{bm}
\usepackage[linkcolor=blue,colorlinks=true]{hyperref}
\catcode`\@=11 \@addtoreset{equation}{section}

\catcode`\@=12

\usepackage{stmaryrd}
\allowdisplaybreaks

\usepackage[labelformat=simple]{subcaption}

\usepackage{enumitem}
\usepackage{url}

\usepackage{tikz}
\usepackage[framemethod=tikz]{mdframed}
\usetikzlibrary{patterns}
\usepackage[normalem]{ulem}
\usepackage{cancel}

\newtheorem{Theorem}{Theorem}[section]

\newtheorem{Lemma}[Theorem]{Lemma}

\newtheorem{Definition}[Theorem]{Definition}

\newtheorem{Remark}[Theorem]{Remark}

\newcommand{\Up}{{\rm Up}}
\newcommand{\Fup}{F_h^{\alpha}}

\newcommand{\vx}{x}
\newcommand{\vrh}{\vr_h}

\newcommand{\vrhup}{\vrh^{\rm up}}
\newcommand{\vrhdown}{\vrh^{\rm down}}
\newcommand{\vthup}{\vth^{\rm up}}
\newcommand{\vthdown}{\vth^{\rm down}}

\newcommand{\vuh}{\bm{u}_h}
\newcommand{\vu}{\bm{u}}

\newcommand{\vth}{\vt_h}
\newcommand{\vthout}{\vth^{\rm out}}
\newcommand{\vthB}{\vt_{B,h}}
\newcommand{\vtB}{\vt_{B}}
\newcommand{\uih}{u_{i,h}}

\newcommand{\Du}{\bD(\vu)}
\newcommand{\Dhuh}{\bD_h(\vuh)}

\newcommand{\br}{ \nonumber \\ }

\newcommand{\RE}[2]{R_E\left((#1)\mid(#2)\right)}
\newcommand{\REH}[2]{\mathbb{E}_{\cal H}\big((#1)|(#2) \big)}

\DeclareMathOperator{\dist}{dist}

\newcommand{\tor}{\mathbb{T}^d}

\newcommand{\bfv}{\bm{v}}

\newcommand{\bfphi}{\boldsymbol{\Phi}}

\newcommand{\Pim}{\Pi_\mathcal{T}}

\newcommand{\Piw}{\Pi_W}
\newcommand{\Piwi}{\Pi_W^{(i)}}

\newcommand{\difuh}{\bS_h:\Gradh \vuh }

\newcommand{\sumi}{\sum_{i=1}^d}

\newcommand{\ds}{\,{\rm d}S_x}

\newcommand{\grid}{{\cal T}_h}

\newcommand{\TS}{\Delta t}

\newcommand{\Div}{{\rm div}_x}
\newcommand{\Grad}{\nabla_x}
\newcommand{\Divh}{{\rm div}_h}
\newcommand{\Gradh}{\nabla_h}

\newcommand{\Laph}{\Delta_h}

\newcommand{\Gradd}{\nabla_\faces}
\newcommand{\Gradedge}{\Gradd}

\newcommand{\Divmesh}{{\rm div}_{\mathcal{T}}}

\newcommand{\pdedgei}{\eth_ \faces^{(i)}}

\newcommand{\pdmeshi}{\eth_{\cal T}^{(i)}}

\newcommand{\co}[2]{{\rm co}\{ #1 , #2 \}}
\newcommand{\EB}[1]{E_B\left( #1 \right)}

\newcommand{\Echi}[1]{E_{\chi} \left( #1 \right)}
\newcommand{\Ov}[1]{\overline{ #1 } }
\newcommand{\avs}[1]{ \left\{\hspace{-3pt}\left\{ #1 \right\}\hspace{-3pt} \right\} }

\newcommand{\aleq}{\stackrel{<}{\sim}}
\newcommand{\ageq}{\stackrel{>}{\sim}}
\newcommand{\Un}[1]{\underline{#1}}
\newcommand{\vr}{\varrho}

\newcommand{\tvr}{\widetilde \vr}
\newcommand{\tvu}{{\widetilde \vu}}

\newcommand{\tvt}{\widetilde \vt}
\newcommand{\tp}{\widetilde p}
\newcommand{\ts}{\widetilde s}

\newcommand{\vt}{\vartheta}
\newcommand{\Ovt}{\Ov{\vt}}
\newcommand{\Uvt}{\underline{\vt}}

\newcommand{\vm}{\bm{m}}
\newcommand{\ve}{\bm{e}}

\newcommand{\vn}{\bm{n}}

\newcommand{\vc}[1]{{\bm #1}}

\newcommand{\dx}{\,{\rm d} {x}}

\newcommand{\dt}{{\rm d} t }

\newcommand{\jump}[1]{\left\llbracket#1\right\rrbracket}

\newcommand{\Bigabs}[1]{\Big| #1\Big|}
\newcommand{\abs}[1]{{\left| #1 \right|}}
\newcommand{\Abs}{\abs}
\newcommand{\biggabs}[1]{\bigg| #1\bigg|}

\newcommand{\norm}[1]{\left\lVert#1\right\rVert}

\newcommand{\dxdt}{\dx \dt}

\newcommand{\intfaces}[1]{\int_{\faces}{ #1 \ds}}
\newcommand{\intfacesB}[1]{\int_{\faces}{ \left( #1\right) \ds}}

\newcommand{\Of}{\Omega^f}
\newcommand{\Os}{\Omega^s}
\newcommand{\Oc}{\Omega^C}
\newcommand{\OI}{\Omega^I}
\newcommand{\OO}{\Omega^O}
\newcommand{\Osh}{\Os_h}
\newcommand{\Ofh}{\Of_h}
\newcommand{\Och}{\Oc_h}
\newcommand{\OIh}{\OI_h}
\newcommand{\OOh}{\OO_h}

\newcommand{\intTd}[1]{\int_{\tor} #1 \dx}
\newcommand{\intTdB}[1]{\int_{\tor}\left( #1 \right)\dx}
\newcommand{\intOs}[1]{\int_{\Os} #1 \dx}
\newcommand{\intOsB}[1]{\int_{\Os} \left(#1\right) \dx}
\newcommand{\intOf}[1]{\int_{\Of} #1 \dx}
\newcommand{\intOfB}[1]{\int_{\Of} \left( #1 \right) \dx}
\newcommand{\intOsh}[1]{\int_{\Osh} #1 \dx}
\newcommand{\intOshB}[1]{\int_{\Osh} \left( #1 \right) \dx}

\newcommand{\intOch}[1]{\int_{\Och} #1 \dx}
\newcommand{\intOchB}[1]{\int_{\Och}\left( #1 \right)\dx}

\newcommand{\intTau}[1]{\int_0^\tau #1 \dt}

\newcommand{\intTauTd}[1]{\int_0^\tau \int_{\tor} #1 \dxdt}
\newcommand{\intTauTdB}[1]{\int_0^\tau \int_{\tor} \left( #1\right) \dxdt}

\newcommand{\intTauOch}[1]{\int_0^\tau \int_{\Och} #1 \dxdt}

\newcommand{\intTauOOh}[1]{\int_0^\tau \int_{\OOh} #1 \dxdt}

\newcommand{\intTauOf}[1]{\int_0^\tau \int_{\Of} #1 \dxdt}
\newcommand{\intTauOfB}[1]{\int_0^\tau \int_{\Of}\left( #1 \right)\dxdt}
\newcommand{\intTauOs}[1]{\int_0^\tau \int_{\Os} #1 \dxdt}
\newcommand{\intTauOsB}[1]{\int_0^\tau \int_{\Os} \left(#1 \right) \dxdt}
\newcommand{\intTauOfh}[1]{\int_0^\tau \int_{\Ofh} #1 \dxdt}
\newcommand{\intTauOsh}[1]{\int_0^\tau \int_{\Osh} #1 \dxdt}

\newcommand{\intTauOshB}[1]{\int_0^\tau \int_{\Osh} \left( #1 \right) \dxdt}

\newcommand{\intTTd}[1]{\int_0^T \int_{\tor} #1 \dxdt}

\newcommand{\intTOf}[1]{\int_0^T \int_{\Of} #1 \dxdt}
\newcommand{\intTOs}[1]{\int_0^T \int_{\Os} #1 \dxdt}

\newcommand{\intn}{\int_{0}^{t^{n+1}}}
\newcommand{\intnOshB}[1]{\int_0^{t^{n+1}} \int_{\Osh} \left( #1 \right) \dxdt}

\newcommand{\vv}{\bm{v}}

\newcommand{\ep}{\varepsilon}

\newcommand{\R}{\mathbb{R}}
\newcommand{\I}{\mathbb{I}}
\renewcommand{\S}{\bS (\Du)}

\def\softd{{\leavevmode\setbox1=\hbox{d}%
  \hbox to 1.05\wd1{d\kern-0.4ex{\char039}\hss}}}

\definecolor{Cgrey}{rgb}{0.85,0.85,0.85}
\definecolor{Cblue}{rgb}{0.50,0.85,0.85}
\definecolor{Cred}{rgb}{1,0,0}
\definecolor{fancy}{rgb}{0.10,0.85,0.10}
\definecolor{forestgreen}{rgb}{0.13, 0.55, 0.13}

\newcommand{\cblue}{\color{blue}}

\newcommand{\pd}{\partial}
\newcommand{\pdt}{\pd_t}
\newcommand{\Hc}{\mathcal{H}_{\tvt}}
\newcommand{\hHc}{\mathcal{H}_{\hvt}}
\newcommand{\penl}{{\varepsilon}}
\newcommand{\faces}{\mathcal{E}}
\newcommand{\facesi}{\faces _i}
\newcommand{\facesK}{\faces(K)}

\newcommand{\bD}{\mathbb D}
\newcommand{\bS}{\mathbb S}
\newcommand{\bI}{\mathbb I}

\newcommand{\tbS}{\widetilde{\bS}}

\newcommand{\muh}{h^\alpha}

\newcommand{\hvt}{\Theta}
\newcommand{\hvth}{\Theta_h}

\newcommand{\myTauangle}[1]{\left\langle \mathcal{V}_{\tau, x};\, #1\right\rangle}
\newcommand{\mytangle}[1]{\left\langle \mathcal{V}_{t, x};\, #1\right\rangle}

\newcommand{\myZangle}[1]{\left\langle \mathcal{V}_{0, x}; \,#1\right\rangle}

\allowdisplaybreaks

\begin{document}

\title{Penalty method for the Navier--Stokes--Fourier system with Dirichlet boundary conditions: convergence and error estimates}
\author{
M\'aria Luk\'a\v{c}ov\'a-Medvi\softd ov\'a\thanks{The work of M.L. and Y. Y. was supported by the Deutsche Forschungsgemeinschaft (DFG, German Research Foundation) - Project number 233630050 - TRR 146.
M.L. is grateful to the Gutenberg Research College and Mainz Institute of Multiscale Modelling for supporting her research.
\quad $^{\spadesuit}$The work of B.S. was funded by the
Czech Science Foundation (GA\v CR), Grant Agreement
21--02411S. 
}
\and Bangwei She$^{\spadesuit}$
\and Yuhuan Yuan$^{\dag, *}$
}

\date{}

\maketitle

\vspace{-0.6cm}




\centerline{$^*$Institute of Mathematics, Johannes Gutenberg-University Mainz}
\centerline{Staudingerweg 9, 55 128 Mainz, Germany}
\centerline{lukacova@uni-mainz.de}

\medskip
\centerline{$^\spadesuit$Academy for Multidisciplinary studies, Capital Normal University}
\centerline{ West 3rd Ring North Road 105, 100048 Beijing, P. R. China}
\centerline{bangweishe@cnu.edu.cn}

\medskip
\centerline{$^\dag$School of Mathematics, Nanjing University of Aeronautics and Astronautics}
\centerline{Jiangjun Avenue No. 29, 211106 Nanjing, P. R. China}
\centerline{yuhuanyuan@nuaa.edu.cn}

\begin{abstract}
We study the convergence and error estimates of a finite volume method for the compressible Navier--Stokes--Fourier system with Dirichlet boundary conditions.
Physical fluid domain is typically smooth and needs to be approximated by a polygonal computational domain. This leads to domain-related discretization errors, the so-called variational crimes.
To treat them efficiently we embed the fluid domain into a large enough cubed domain, and propose a finite volume scheme for the corresponding domain-penalized problem.
Under the assumption that the numerical density and temperature are uniformly bounded, we derive the ballistic energy inequality, yielding a priori estimates and the consistency of the penalization finite volume approximations. Further, we show that the numerical solutions converge weakly to a generalized, the so-called dissipative measure-valued, solution of the corresponding Dirichlet problem. If a strong solution exists, we prove that our numerical approximations converge strongly with the rate $1/4$. Additionally, assuming uniform boundedness of the approximate velocities, we obtain global existence of the strong solution.
In this case we prove that the numerical solutions converge strongly to the strong solution with the optimal rate $1/2$.
\end{abstract}



\section{Introduction}

Motivated by physically relevant problems we consider a compressible, viscous, heat-conductive fluid confined to a smooth bounded domain $\Of \subset \R^d, \ d=2,3,$ with the prescribed data on a boundary. Such flows are governed by an open
Navier--Stokes--Fourier (NSF) system
\begin{subequations}\label{pde}
\begin{align}
\pd_t \vr + \Div (\vr \vu )& = 0, \label{pde_d}
\\
\pd_t (\vr \vu) + \Div (\vr \vu \otimes \vu ) + \Grad p & = \Div \S , \label{pde_m}
\\
\pd_t (\vr e) + \Div ( \vr e\vu) + \Div \vc{q}(\Grad \vt)& = { \S: \Grad \vu } -p \Div \vu \label{pde_e},
\end{align}
that is equipped with the initial data
\begin{equation} \label{pde_ic}
(\vr ,\vu,\vt)(0,\cdot) =(\vr _0, \vu _0,\vt _0)
\end{equation}
and the Dirichlet boundary conditions. In this paper we consider
\begin{equation} \label{pde_bc}
	\vu|_{\pd \Of} = 0, \quad \vt|_{\pd \Of} = \vtB(t,x).
\end{equation}
Here, $\vr, \vu, \vt, p$ and $e$ are the fluid density, velocity, absolute temperature, pressure and internal energy, respectively, $\vtB$ represents a given temperature at the boundary. Further, $\S$ is the viscous stress tensor
\begin{align}
\S= 2\mu \Du + \lambda \Div \vu \I \quad \mbox{with} \quad \Du=\frac{\Grad \vu + \Grad^T \vu}{2}, \ \ \Div \vu = \rm{tr}(\Du),\ \ \mu > 0, \ \ \lambda \geq 0
\end{align}
and $\vc{q}(\Grad \vt)$ is the heat flux given by Fourier's law
\begin{equation}
\vc{q}(\Grad \vt) = - \kappa \Grad \vt, \quad \kappa > 0,
\end{equation}
where the constants $\mu, \lambda $ are the viscosity coefficients, the constant $\kappa$ is the thermal conductivity coefficient, and $\mbox{tr}$ denotes the trace of a matrix.

To close system \eqref{pde} we need the state equations to interrelate the thermodynamic variables.
Hereafter we consider the standard pressure law of a perfect gas
\begin{equation}\label{perfect-gas}
p = (\gamma -1) \vr e, \quad e=c_v\vt,
\end{equation}
\end{subequations}
where $\gamma > 1$ is the adiabatic coefficient and $c_v = \frac1{\gamma-1} > 0 $ is the specific heat at constant volume.
Throughout the paper we suppose that the initial data satisfy
\begin{equation} \label{ic}
 \vr_0 \geq \Un{\vr} > 0,~ \vt_0 > 0, ~ \vr_0 \in L^\infty(\Of), ~\vm_0 \in L^\infty(\Of;\R^d), ~ \vt_0 \in L^\infty(\Of).
\end{equation}

\medskip
It is well-known that the Dirichlet problem \eqref{pde} admits a local strong solution if the initial/boundary data are smooth enough, see \cite{Valli:1982b,Valli:1982a,Valli-Zajaczkowski:1986} and references therein.
Recently, Chaudhuri \cite{Chaudhuri:2022} showed the dissipative measure valued (DMV)--strong uniqueness principle for open NSF system \eqref{pde} with the Dirichlet boundary conditions. More precisely, it means that a strong solution is stable in the class of very weak, the so-called DMV, solutions.
Further, Basari\'c et al.~\cite{BFM} proved conditionally regularity for compressible, viscous, and heat-conductive fluids with the perfect gas state equation. Thus, system \eqref{pde} admits a global strong solution if a global weak solution of the system \eqref{pde} is bounded. These results strongly rely on the existence of global weak or DMV solutions. In this context we note that (unconditional) global existence of a solution to the Navier--Stokes--Fourier system with the perfect gas law is only available in the class of DMV solutions, the corresponding result in the class of weak solutions is still an open problem. We refer a reader to Feireisl, Novotn\'y~\cite{FeiNov}, where global existence of a weak solution was proved for a specific
equation of state, see
also Chaudhuri and Feireisl \cite{Chaudhuri-Feireisl:2022} for the corresponding weak--strong uniqueness principle.

 The main goal of this paper is to show the convergence as well as the error estimates of finite volume (FV) approximations of the NSF system \eqref{pde}--\eqref{ic}. As a consequence, we also prove the global existence of DMV solutions for the NSF system with the Dirichlet boundary conditions, which was not available in literature. We point out that the obtained DMV solutions are thermodynamically consistent, i.e.~they satisfy the second law of thermodynamics, which is expressed by the following entropy inequality
\begin{equation} \label{pde_s}
\partial_t (\vr s ) + \Div (\vr s \vu) + \Div \left( \frac{\vc{q}(\Grad \vt) }{\vt} \right) \geq \frac{1}{\vt} \left( \S: \Du - \frac{\vc{q}(\Grad \vt) \cdot \Grad \vt}{\vt} \right)
\end{equation}
holding in weak sense. Here $s$ denotes the physical entropy for a perfect gas
\begin{equation}\label{eq_s}
s(\vr, \vt)= \log \left(\frac{\vt^{c_v}}{\vr} \right) \qquad \mbox{ for } \vr > 0, \ \vt >0.
\end{equation}

\subsection{Penalized Navier--Stokes--Fourier system on $\tor$}

The penalty approach is a popular modelling tool when the boundary of a physical domain has a complicated
structure and its approximation by polygons is computationally expensive. An efficient way to overcome this difficulty is to transform geometry to penalty terms that arise in the governing equations as additional forcing terms. The latter can be seen as the Lagrange multiplier terms. In this context we refer to the Lagrange multiplier based fictitious domain method \cite{Glowinski-Pan-Periaux:1994,Glowinski-Pan-Periaux:1994a,Hyman:1952} and the immersed boundary method \cite{Peskin:1972,Peskin:2002} that have been applied to incompressible Navier--Stokes equations, see also \cite{angot}. In \cite{Hesthaven1, Hesthaven2} the penalization of boundary conditions was discussed for compressible NSF system in the context of a spectral method. For elliptic boundary problems rigorous error analysis of penalization methods were presented in \cite{Maury:2009, Saito-Zhou:2015,Zhang:2006,Zhou-Saito:2014}.

In our recent paper~\cite{LSY_penalty} we have successfully applied a penalty method for the barotropic Navier--Stokes system. Inspired by \cite{LSY_penalty} our aim is to extend and generalize the penalty method to the NSF system. This is a nontrivial task due to the fact that we need to deal with an additional equation for temperature evolution. Moreover, the Dirichlet boundary condition for temperature poses an additional difficulty. {\em The novelty of our approach} lies in the use of {\em the ballistic energy inequality} in addition to the entropy inequality. They both yield desired a priori estimates that allow us to study the convergence of the penalty method and the existence of a global DMV solution.

Specifically, embedding the fluid domain $\Of$ into a large enough torus $\tor$, see Figure \ref{figD}, a penalized NSF system reads
\begin{subequations}\label{ppde}
\begin{align}
\partial_t \vr + \Div (\vr \vu) &= 0, \label{ppde_d}	
\\
\partial_t (\vr \vu) + \Div (\vr \vu \otimes \vu) + \Grad p(\vr) &= \Div \bS( \Grad \vu) - \frac{ \mathds{1}_{\Os} }{\penl} \vu, \label{ppde_m}
\\
\pd_t (\vr e) + \Div ( \vr e\vu) + \Div \vc{q}(\Grad \vt)& = { \S: \Grad \vu } -p \Div \vu- \frac{ \mathds{1}_{\Os} }{\penl} (\vt - \vtB). \label{ppde_e}
\end{align}
The initial data are given as
\begin{equation}\label{ppde_ic}
\begin{aligned}
&(\tvr,\tvu,\tvt)(0, \cdot) := (\tvr_{0},\tvu_0,\tvt_0) =
\begin{cases}
 (\vr_0^s,0,\vtB) & \mbox{if} \ \vx \in \Os:=\tor \setminus \Of, \\
(\vr_0,\vu_0,\vt_0) &\mbox{if} \ \vx \in \Of,
\end{cases}
\\
& \tvr_{0} > 0, ~ \tvt_0 > 0 \ \ \ \mbox{satisfying the periodic boundary condition}. 
\end{aligned}
\end{equation}
\end{subequations}
The corresponding entropy inequality reads
\begin{align}\label{ppde_s}
\partial_t (\vr s ) + \Div (\vr s \vu) + \Div \left( \frac{\vc{q}(\Grad \vt) }{\vt} \right) \geq \frac{1}{\vt} \left( \S: \Du - \frac{\vc{q}(\Grad \vt) \cdot \Grad \vt}{\vt} - \frac{ \mathds{1}_{\Os} }{\penl} (\vt - \vtB)\right).
\end{align}
Here, $0 < \penl << 1$ is the penalty parameter, $ \mathds{1}_{\Os} (\vx)$ is the characteristic function defined by
 \begin{equation*}
 \mathds{1}_{\Os} (\vx)=
\begin{cases}
1 & \mbox{if} \ \vx \in \Os ,\\
0 & \mbox{if} \ \vx \in \Of
\end{cases}
\end{equation*}
and $\tvr_0$ is an extension of $\vr_0$.
Here and hereafter, we suppose that $\vtB$ which is defined on $\tor$ satisfies:
\begin{align*}
 \vtB \in W^{2, \infty}((0,T) \times \mathbb{T}^d), \quad	\inf \vtB > 0.
\end{align*}

\begin{figure}[!h]
\centering
\begin{tikzpicture}[scale=0.6]
\draw[color=blue, very thick] ( -5,-3.2 ) rectangle ( 5, 3.2); 
\path node at (-4,2) { \cblue \huge $\tor$};
\draw[very thick] (0,0) ellipse (4 and 2.5 );
\path node at (0,0) { \huge $\Of$};
\end{tikzpicture}
\caption{A fluid domain $\Of$ embedded into a torus $\tor$.}
\label{figD}
\end{figure}
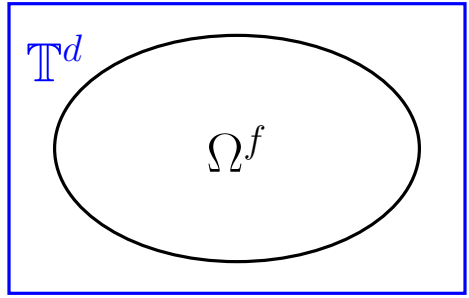

\

The rest of the paper is organized as follows. In Section~\ref{sec_scheme} we introduce notations and then present a FV method for the approximation of the penalized problem \eqref{ppde}.
In Section~\ref{sec_sta} and Section~\ref{sec_con} we study the stability and consistency of our numerical method, respectively.
Section~\ref{sec_convergence} and Section~\ref{sec_EE} build the heart of the paper. We present here the main results: the weak (resp. strong) convergence of numerical approximations towards the DMV (resp. strong) solution of the Dirichlet problem \eqref{pde}, cf.~Theorems \ref{THM2} (resp. \ref{THM3}) as well as the error estimates between the numerical approximations and the strong solution of the Dirichlet problem,
cf.~Theorems \ref{thm_EEB} and \ref{thm_EEB_1}. Technical details of proofs are presented in Appendices~A. -- D.

\section{Numerical scheme}\label{sec_scheme}
 In this section we generalize the FV method proposed in \cite{FLMS_FVNSF,LSY_penalty} to approximate the penalized problem \eqref{ppde}.

\subsection{Notations}
We start  by introducing notation $a \aleq b$. It means that there exists a generic positive constant $C$ independent of the time step $\TS$, mesh size $h$ and penalty parameter $\penl$ such that $a\leq C\cdot b$. Further we write $a \approx b$ if $a \aleq b$, $b \aleq a$, $a,b \in \R$.
\paragraph{\bf Mesh.} Let $\grid$ be a uniform structured (square for $d=2$ or cuboid for $d=3$) mesh of $\tor$ with $h\in(0,1)$ being the mesh size.
We denote by $\faces$ the set of all faces of $\grid$ and by $\facesi$, $i=1,\dots, d$, the set of all faces that are orthogonal to $\ve_i$ -- the basis vector of the canonical system.
Moreover, we denote by $\facesK$ the set of all faces of a generic element $K\in\grid$. Then, we write $\sigma= K|L$ if $\sigma \in \faces$ is the common face of neighbouring elements $K$ and $L$. Further, we denote by $\vx_K$ and $|K|=h^d$ (resp. $\vx_\sigma$ and $|\sigma|=h^{d-1}$) the center and the Lebesgue measure of an element $K\in\grid$ (resp. a face $\sigma \in \faces$), respectively.

\begin{Definition}[Artificial splitting of the mesh]
\label{def_ES2}
\ \hfill
\begin{itemize}
\item We split the mesh $\grid$ into two parts: the part consisting of all elements inside the physical domain $\Of$ (denoted as $\Ofh$) and its complement (denoted as $\Osh$):
\begin{align*}
\Ofh = \left\{ K ~ \big| ~ K\subset \Of \right\} \quad \mbox{and}\quad \Osh = \grid \setminus \Ofh,
\end{align*}
cf. Figure \ref{fig:domain} $(a)$.

\begin{figure}[htbp]
	\centering
	\begin{subfigure}{0.33\textwidth}
		\includegraphics[width=\textwidth]{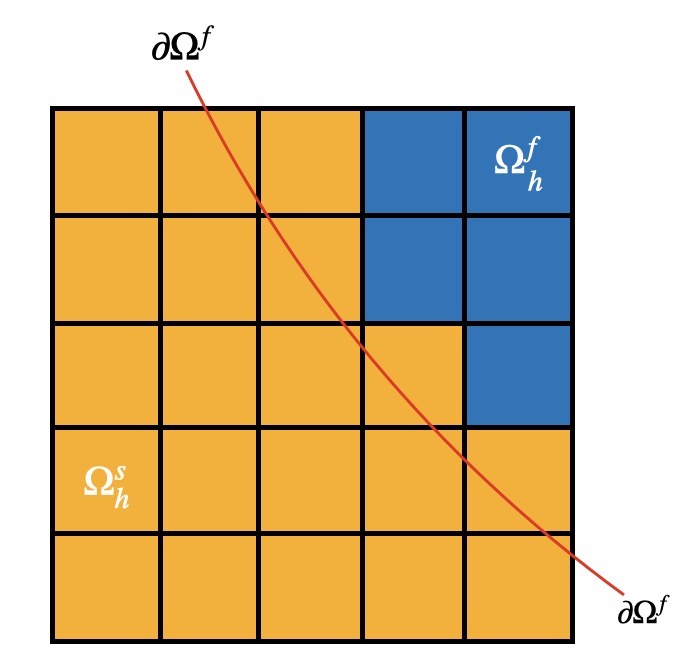}
		\caption{ }
	\end{subfigure}	\quad
	\begin{subfigure}{0.32\textwidth}
		\includegraphics[width=\textwidth]{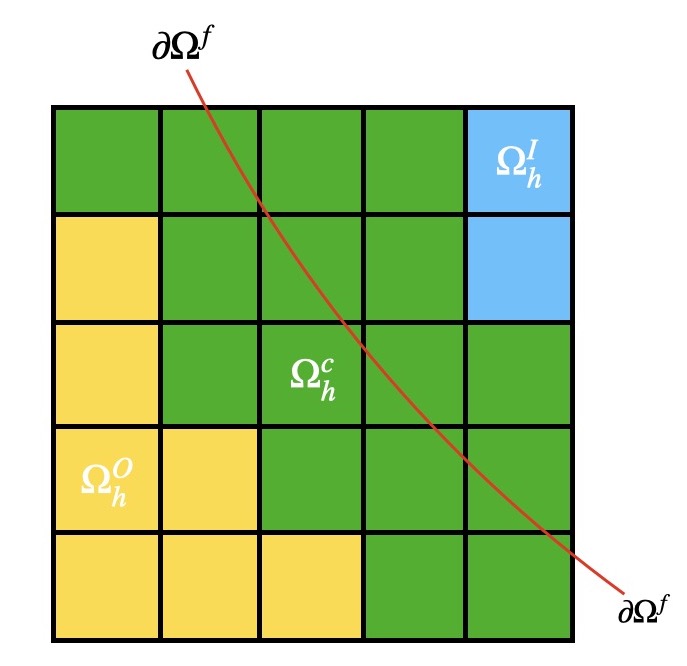}
		\caption{ }
	\end{subfigure}	
	\caption{\small{ Zoom-in of domain splitting. }}\label{fig:domain} 
\end{figure}

\item We split the mesh $\grid$ into three pieces. $\Och$ denotes the area containing the neighbourhood of the fluid boundary $\pd \Of$
\begin{align*}
\Och = \left\{ K \mid \cup_{L\cap K\neq \emptyset} L \cap \partial \Of \neq \emptyset \right\}.
\end{align*}
The inner domain $\OIh$ and the outer domain $\OOh$ are given as
\[\OIh \colon = \Of \setminus \Och \quad \mbox{ and }\quad \OOh \colon =\Os \setminus \Och ,\]
see Figure \ref{fig:domain} $(b)$. We note that this special spitting shall be useful in the analysis.
\end{itemize}
\end{Definition}

\begin{Remark}\label{rmk}
According to the definitions of splittings, we have
\begin{equation}\label{EXTE0}
 \OIh \subset \Ofh \subset \Of,\quad
\OOh \subset \Os \subset \Os_h , \quad
 \Os_h \setminus \Os \subset \Os_h \setminus \OOh \subset \Och.
\end{equation}
It is obvious that $\dist(\partial \Of, \Ofh) \aleq h$, where the constant $C$ depends on the geometry of the fluid domain $\Of$.
Furthermore, it holds that
\begin{equation}\label{EXTD}
 |\Of \setminus \Ofh| \leq \abs{\Och} \aleq h,\quad \abs{\Osh \setminus \Os} \leq \abs{\Os_h \setminus \OOh} \leq \abs{\Och} \aleq h.
\end{equation}
\end{Remark}

\paragraph{\bf Dual mesh.} For any $\sigma=K|L$, we define a dual cell $D_\sigma := D_{\sigma,K} \cup D_{\sigma,L}$, where $D_{\sigma,K}$ is defined as
\begin{align*}
D_{\sigma,K} = \left\{ x \in K \mid x_i\in\co{(x_K)^{(i)}}{(x_\sigma)^{(i)}}\right\} \mbox{ for any } \sigma \in \facesi, \; i=1,\ldots,d,
\end{align*}
where the superscript $(i)$ denotes the $i$-th component of a vector and
\begin{equation}\label{eqco}
\co{A}{B} \equiv [ \min\{A,B\} , \max\{A,B\}].
\end{equation}
We refer to Figure~\ref{fig:mesh} for a two-dimensional illustration of a dual cell.

\begin{figure}[!h]
\centering
\begin{tikzpicture}[scale=1.0]
\draw[-,very thick](0,-2)--(5,-2)--(5,2)--(0,2)--(0,-2)--(-5,-2)--(-5,2)--(0,2);

\draw[-,very thick, green=90!, pattern=north east lines, pattern color=green!30] (0,-2)--(2.5,-2)--(2.5,2)--(0,2)--(0,-2);
\draw[-,very thick, blue=90!, pattern= north west lines, pattern color=blue!30] (0,-2)--(0,2)--(-2.5,2)--(-2.5,-2)--(0,-2);

\path node at (-3.5,0) { $K$};
\path node at (3.5,0) { $L$};
\path node at (-2.5,0) {$ \bullet$};
\path node at (-2.8,-0.3) {$ \vx_K$};
\path node at (2.5,0) {$\bullet$};
\path node at (2.7,-0.3) {$ \vx_L$};
\path node at (0,0) {$\bullet$};
\path node at (0.3,-0.3) {$ \vx_\sigma$};

\path (-0.4,0.8) node[rotate=90] { $\sigma=K|L$};
 \path (-1.5,1.4) node[] { $D_{\sigma,K}$};
 \path (1.5,1.4) node[] { $D_{\sigma,L}$};
 \end{tikzpicture}
\caption{Dual mesh $D_\sigma = D_{\sigma,K} \cup D_{\sigma,L}$}
 \label{fig:mesh}
\end{figure}
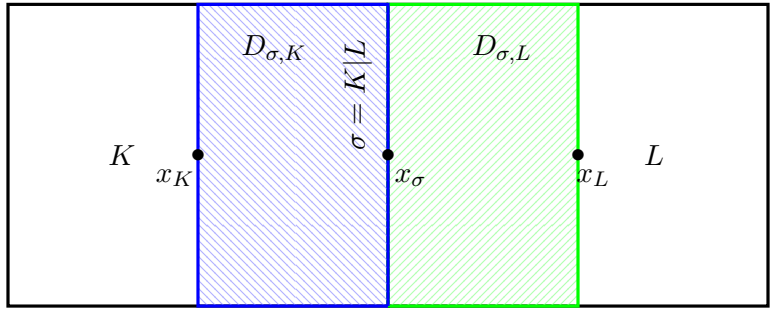

\paragraph{\bf Function space.}
We denote by $Q_h$ the set of all piecewise constant functions on the grid $\grid$ and $Q_h^m$ as the corresponding $m$-dimensional function space, $m \in \mathbb{N}$. We associate $Q_h$ with the projection operator
\begin{align*}
 \Pim \phi (x) = \sum_{K \in \grid} \frac{\mathds{1}_{K}(x)}{|K|} \int_K \phi \dx \quad
 \mbox{for any } \phi\in L^1(\tor).
\end{align*}
Moreover, for $v \in Q_h$ we introduce the average and jump operators at the interface:
\[
\avs{v}(x) = \frac{v^{\rm in}(x) + v^{\rm out}(x) }{2},\ \ \ \
\jump{ v }(x) = v^{\rm out}(x) - v^{\rm in}(x), \]
where
\[
v^{\rm out}(x) = \lim_{\delta \to 0+} v(x + \delta \vc{n}),\ \ \ \
v^{\rm in}(x) = \lim_{\delta \to 0+} v(x - \delta \vc{n})
\]
whenever $x \in \sigma \in \faces$ and $\vc{n}$ is the outer normal vector to $\sigma$.
Moreover, we define the upwind and downwind quantities of $v \in Q_h$ associated to the velocity field $\vu\in Q_h^d$ at a generic face $\sigma$
\begin{equation*}
(v^{\rm up}, v^{\rm down}) =
\begin{cases}
( v^{\rm in},v^{\rm out}) & \mbox{if} \ \avs{\bm{u}} \cdot \vc{n} \geq 0, \\
(v^{\rm out},v^{\rm in}) & \mbox{if} \ \avs{\bm{u}} \cdot \vc{n} < 0.
\end{cases}
\end{equation*}

\paragraph{Discrete differential operators.}
We define $\Gradh$, $\Divh$ and $\bD_h$ on the mesh $\mathcal{T}$ as follows:
\begin{equation*}
\begin{aligned}
\left( \Gradh r_h\right)_K &= \frac{|\sigma|}{|K|} \sum_{\sigma \in \facesK} \avs{r_h} \vc{n}, \quad \quad&\Gradh r_h(x) &= \sum_{K \in \grid} \left( \Gradh r_h\right)_K \mathds{1}_K{(x)},
 \quad \quad r_h \in Q_h,
\\
 \left( \Divh \bfv_h \right)_K &=
 \frac{|\sigma|}{|K|} \sum_{\sigma \in \facesK} \avs{\bfv_h} \cdot \vc{n} , & \Divh \bfv_h(x) &= \sum_{K\in\grid} \left( \Divh \bfv_h \right)_K \mathds{1}_K{(x)} , 
 \\
 \Gradh \bfv_h & = \left( \Gradh v_{1,h}, \dots, \Gradh v_{d,h}\right)^T, &
\bD_h (\bfv_h) &= \frac{\Gradh \bfv_h +\Gradh^T \bfv_h}{2}, 
\quad \quad \bfv_h= (v_{1,h},\ldots,v_{d,h})^T \in Q_h^d.
\end{aligned}
\end{equation*}
It holds $\Divh \bfv_h = \mbox{tr} (\Gradh \bfv_h) = \mbox{tr} (\bD_h (\bfv_h))$.

Furthermore, for any $r_h \in Q_h$ we define another discrete gradient operator $ \Gradd$ on the dual mesh $D_{\sigma}$
\begin{equation*}
 \Gradd r_h(x) = \sum_{\sigma\in\faces} \left(\Gradd r_h \right)_{\sigma}\mathds{1}_{D_\sigma}{(x)}, \quad \left(\Gradd r_h\right) _{\sigma} =\frac{\jump{r_h} }{ h } \vc{n}
\end{equation*}
and the discrete Laplace operator $\Delta_h$ on the mesh $\mathcal{T}$
\begin{equation*}
\left(\Delta_h r_h\right)_K = \frac{|\sigma|}{|K|} \sum_{\sigma \in \facesK} \frac{ \jump{r_h} }{h}, \quad \Delta_h r_h(x) = \sum_{K \in \grid} \left(\Delta_h r_h\right)_K \mathds{1}_K{(x)}.
\end{equation*}

\paragraph{Time discretization.}
Given a time step $\TS>0$ we divide the time interval $[0,T]$ into $N_T=T/\TS$ uniform parts, and denote $t^k= k\TS.$
Moreover, we denote by $v_h^k \in Q_h$ a piecewise constant approximation of the function $v$ at time $t^k$.
Further, we denote by $v_h(t)\in L_{\TS}(0,T; Q_h)$ the piecewise constant approximation in time, such that
\begin{align*}
 v_h(t, \cdot) =v_h^0 \ \mbox{ for } \ t < \Delta t, \quad v_h(t,\cdot)=v_h^k \ \mbox{ for } \ t\in [k\TS,(k+1)\TS), \ \ k=1,\dots, N_T.
\end{align*}
Furthermore, we define the discrete time derivative by a backward difference formula
\[
 D_t v_h(t) = \frac{v_h (t) - v_h^\triangleleft}{\TS} \quad \mbox{with} \quad v_h^\triangleleft = v_h (t - \TS).
\]

\subsection{Finite volume method}
We are now ready to propose an FV method with an upwind numerical flux.
\begin{Definition}[Finite volume method]
Let the initial data \eqref{pde_ic} be extended by $(\tvr_0, \tvu_0, \tvt_0)$ as in \eqref{ppde_ic} and let $(\vrh^0,\vuh^0, \vth^0) :=\Pim (\tvr_0, \tvu_0, \tvt_0)$ and $\vthB = \Pim \vtB$.

 \vspace{0.1cm}

We say that
$(\vrh^{\penl},\vuh^{\penl}, \vth^{\penl}) \in L_{\TS}(0,T; Q_h^{d+2})$ is a finite volume approximation of the penalized problem \eqref{ppde} if it solves the following system of algebraic equations
\begin{subequations}\label{VFV}
\begin{align}\label{VFV_D}
&\intTd{ D_t \vrh^{\penl} \phi_h} - \intfaces{ \Fup (\vrh^{\penl} ,\vuh^{\penl} )
\jump{\phi_h} } = 0, \hspace{4.5cm} \mbox{for all}\ \phi_h \in Q_h, \\ \label{VFV_M}
&\intTd{ D_t (\vrh^{\penl} \vuh^{\penl} ) \cdot \bfphi_h } - \intfaces{ \Fup (\vrh^{\penl} \vuh^{\penl} ,\vuh^{\penl} ) \cdot \jump{\bfphi_h} }
+ \frac{1}{\penl}\intOsh{\vuh^{\penl} \cdot \bfphi_h}
 \br
 &\hspace{3cm} + \intTd{ (\bS_h^{\penl} -p_h^{\penl} \bI ) : \bD_h \bfphi_h }
= 0,
\hspace{3cm} \mbox{for all } \bfphi_h \in Q_h^d,
\\ \label{VFV_E}
&c_v\intTd{ D_t (\vrh^{\penl} \vth^{\penl} ) \phi_h } - c_v\intfaces{ \Fup (\vrh^{\penl} \vth^{\penl} ,\vuh )\jump{\phi_h} }
+\intTd{ \kappa \Gradd \vth^{\penl} \cdot \Gradd \phi_h}
 \br
 &\hspace{1cm} + \frac{1}{\penl}\intOsh{ (\vth^{\penl}-\vthB) \phi_h}
= \intTd{ (\bS_h^{\penl} -p_h^{\penl} \bI ) : \Gradh \vuh^{\penl} \phi_h}, \quad \hspace{1.6cm}\mbox{for all}\ \phi_h \in Q_h .
 \end{align}
\end{subequations}
Here $\bS_h^\penl = 2\mu \bD_h \vuh^{\penl} + \lambda \Divh\vuh^{\penl} \bI $ and the numerical flux $\Fup (r_h,\vuh)$ is taken as the diffusive upwind flux
\begin{equation*}
\Fup (r_h,\vuh)
={\Up}[r_h, \vuh] - \muh \jump{ r_h }
\quad \mbox{ with } \Up [r_h, \bm{u}_h] = r_h^{\rm up} \avs{\vu_h} \cdot \vn.
\end{equation*}
 Note that $ \alpha >-1$ is the artificial viscosity parameter.
\end{Definition}

\begin{Remark}
Hereafter, we shall omit the superscript $\penl$ for simplicity if there is no confusion, e.g.,  writing $(\vrh,\vuh,\vth)$ instead of $(\vrh^{\penl},\vuh^{\penl},\vth^{\penl})$.
\end{Remark}

\section{Stability}\label{sec_sta}
Following the Lax equivalence principle \cite{FeLMMiSh} we need to show the stability and consistency of a numerical method in order to prove its convergence. We start with studying the stability of the FV method \eqref{VFV}.
\subsection{Positivity of density and temperature}
\begin{Lemma}[Positivity of density and temperature] \label{lem_p1}
Let $(\vrh ,\vuh ,\vth )$ be a numerical solution of the finite volume method \eqref{VFV}. Then it holds
\begin{align*}
\vrh(t) > 0, \quad \vth(t) > 0 \quad \mbox{for all } \ t\in(0,T).
\end{align*}
Note that $\vrh^0 > 0$ and $\vth^0 > 0$ follow from the positivity of the data that $\tvr^0 > 0$ and $\tvt^0 > 0$.
\end{Lemma}
The positivity of the density and temperature can be derived from the so-called renormalized equations, see Appendix \ref{app-positivity}.

\subsection{Energy balance}
\begin{Lemma}[Energy balance]\label{thm_energy_stability}
Let $(\vrh ,\vuh ,\vth )$ be a numerical solution of the finite volume method \eqref{VFV}. Then it holds
\begin{align}\label{energy_stability}
 D_t \intTd{ \left(\frac{1}{2} \vrh |\vuh |^2 + { c_v \vrh \vth } \right) }
 +\frac{1}{\penl} \intOsh{ \left( |\vuh|^2 + (\vth-\vthB) \right) } = - D_{\rm E},
 \end{align}
where $D_{\rm E} \geq 0$ is the numerical dissipation given by
\begin{align}\label{energy_dissipation}
D_{\rm E} = \frac{\TS}{2} \intTd{ \vrh^\triangleleft|D_t \vuh |^2 } + h^\alpha \intfaces{ \avs{ \vrh } \abs{\jump{\vuh}}^2 }
+ \frac12 \intfaces{ \vrh^{\rm up} |\avs{\vuh } \cdot \vc{n} | \abs{\jump{ \vuh }}^2 }.
\end{align}
\end{Lemma}
\begin{proof}
We start by recalling the kinetic energy balance, cf.~\cite[equation (3.4)]{FLMS_FVNS},
\begin{multline*} 
 D_t \intTd{ \frac{1}{2} \vrh |\vuh |^2 }
+ \intTd{ (\bS_h- p_h\I): \Gradh \vuh } +\frac{1}{\penl} \intOsh{|\vuh|^2}
\\
+ h^\alpha \intfaces{ \avs{ \vrh } \abs{\jump{\vuh}}^2 }
+ \frac{\TS}{2} \intTd{ \vrh^\triangleleft|D_t \vuh |^2 }
+ \frac12 \intfaces{ \vrh ^{\rm up} |\avs{\vuh } \cdot \vc{n} | \abs{\jump{ \vuh }}^2 } =0.
\end{multline*}
Setting $\phi_h=1$ in \eqref{VFV_E} we get
\[
D_t \intTd{ c_v\vrh \vth }+\frac{1}{\penl} \intOsh{(\vth-\vthB)}
= \intTd{(\bS_h- p_h\I): \Gradh \vuh}.
\]
Summing up the above two equations we finish the proof.
\end{proof}

\subsection{Entropy balance}
\begin{Lemma}[Entropy balance]\label{THENT}
Let $(\vrh ,\vuh ,\vth )$ be a numerical solution of the finite volume method \eqref{VFV} and $\phi_h\in Q_h$. Then it holds
\begin{multline}\label{eq_entropy_stability}
 \intTd{ D_t \left(\vrh s_h \right) \phi_h } -
\intfaces{ \Up(\vrh s_h , \vuh ) \jump{\phi_h} }
 + \frac{1}{\penl} \intOsh{(\vth -\vthB)\frac{\phi_h}{\vth}}
\\
- \intTd{ \frac{\phi_h}{\vth } \difuh } + \intTd{\kappa \Gradd\vth \cdot \Gradd\!\!\left(\frac{\phi_h}{\vth }\right)}
 = D_s(\phi_h) + R_{s}(\phi_h),
\end{multline}
where $\bS_h=2\mu | \Dhuh|^2 + \lambda |\Divh \vuh|^2$, $D_s(\phi_h) = \sum_{i=1}^3 D_{s,i}(\phi_h)$ and $R_{s}(\phi_h)$ are given by
\begin{equation}\label{entropy_dissipation}
\begin{aligned}
D_{s,1}(\phi_h) := & \ \TS \intTd{ \left( \frac{ | D_t \vr _h |^2 }{2 \xi _{\vr,h} } +\frac{c_v \vrh^\triangleleft |D_t \vth |^2 }{2 |\xi_{\vt,h} |^2 } \right) \phi_h},
\\
 D_{s,2}(\phi_h) := & \ \frac12 \intfaces{ \phi_h^{\rm down} \abs{ \avs{\vuh}\cdot \vc{n} } \jump{ (\vrh, p_h) } \cdot \nabla_{(\vr,p)}^2 (-\vr s)|_{\vv_1^*}\cdot \jump{ (\vrh, p_h) } },
\\
D_{s,3}(\phi_h) :=
 & \ h^\alpha \intfaces{ \avs{\phi_h} \jump{ (\vrh, p_h) } \cdot \nabla_{(\vr,p)}^2 (-\vr s)|_{\vv_2^*}\cdot \jump{ (\vrh, p_h) }},
 \\
R_{s}(\phi_h) := & \ h^\alpha \intfaces{ \jump{\phi_h} \cdot \bigg( \avs{ \nabla_\vr(-\vrh s_h) } \jump{ \vrh}
	+ \avs{\nabla_p(-\vrh s_h)} \jump{ p_h}\bigg)}
\end{aligned}
\end{equation}
with
$ \xi _{\vr,h}\in \co{\vrh^\triangleleft}{\vrh},\xi _{\vt,h}\in \co{\vth^\triangleleft}{\vth}, \vv_1^*, \vv_2^* \in \co{(\vrh^{\rm in}, p_h^{\rm in})}{(\vrh^{\rm out}, p_h^{\rm out})}$, and $\co{\cdot}{\cdot}$ given in \eqref{eqco}.
Moreover, it holds that $D_s(\phi_h) \geq 0$ for any $ \phi_h \geq 0$.
\end{Lemma}
\begin{proof}
To begin, let us keep in mind that the numerical density and temperature are positive, cf. Lemma~\ref{lem_p1}.
Then, setting $B(\vr)= \vr \log(\vr)$ in the renormalized density equation (\ref{renormalized_density}) we have
\begin{align}\label{total_ener2}
& \intTd{ D_t \left( \vr _h \log(\vrh ) \right) \phi_h}
 - \intfaces{ \Up [ \vrh \log(\vrh ), \vuh ] \jump{ \phi_h } }
 + \intTd{ \vrh \Divh \vuh \phi_h }
\br
=
& - \intTd{ \frac{ \TS}{2 \xi _{\vr,h} } | D_t \vr _h |^2 \phi_h }
- h^\alpha \intfaces{ \jump{ \vrh } \jump{\left(\log(\vrh )+1\right) \phi_h } }
\nonumber \\
& \quad -\intfaces{ |\avs{\vuh } \cdot \vn | \phi_h^{\rm down} \EB{\vrhup|\vrhdown} },
\quad \quad \ \xi _{\vr,h}\in \co{\vrh^\triangleleft}{\vrh},
\end{align}
where $E_f(v_1|v_2) = f(v_1) - f'(v_2)(v_1-v_2) - f(v_2)$.

Moreover, taking the test function in the renormalized internal energy equation \eqref{eq_renormalized_energy} as $\chi(\vt)=\log(\vt)$ we get
\begin{align}\label{eq_renormalized_energy2}
 & c_v \intTd{ D_t \left(\vrh \log(\vth ) \right) \phi_h } -
c_v\intfaces{ \Up (\vrh \log(\vth ), \vuh ) \jump{\phi_h} }
+\intfaces{ \frac{\kappa}{h} \jump{\vth } \jump{\frac{\phi_h}{\vth } } }
\br
= & \intTdB{ \bS_h: \Gradh \vuh \frac{\phi_h}{\vth } - \vrh \Divh \vuh \phi_h } + \frac{c_v \TS}{2} \intTd{ \vrh^\triangleleft \biggabs{\frac{D_t \vth }{\xi_{\vt,h} }}^2 \phi_h }
\br &
- c_v h^\alpha \intfacesB{ \jump{\vrh } \jump{ \left( \log(\vth )-1\right) \phi_h} + \jump{ p_h } \jump{\frac{\phi_h}{\vth }}}-c_v\intfaces{ |\avs{\vuh } \cdot \vn | \phi_h^{\rm down} \vrhup\Echi{\vth^{\rm up}|\vth^{\rm down} }}
\br &
 - \frac{1}{\penl}\intOsh{ \frac{\vth-\vthB}{ \vth} \phi_h },
 \quad \quad \ \xi _{\vt,h}\in \co{\vth^\triangleleft}{\vth}.
\end{align}
Next, subtracting \eqref{total_ener2} from \eqref{eq_renormalized_energy2}, and thanks to the equalities
\begin{equation*}
\nabla_\vr(-\vr s(\vr,p)) = c_v+1 -s \quad \mbox{and} \quad \nabla_p(-\vr s(\vr,p))= - c_v/\vt
\end{equation*}
we obtain
\begin{align*}
& \intTd{ D_t \left(\vrh s_h \right) \phi_h } -
\intfaces{ \Up(\vrh s_h , \vuh ) \jump{\phi_h} }
 + \frac{1}{\penl}\intOsh{ \frac{\vth-\vthB}{ \vth} \phi_h }
\br & \qquad
+\intfaces{ \frac{\kappa}{h} \jump{\vth } \jump{\frac{\phi_h}{\vth } } }
- \intTd{ \frac{\phi_h}{\vth } \difuh }
\br = ~&
\TS \intTd{ \left( \frac{ | D_t \vr _h |^2 }{2 \xi _{\vr,h} } +\frac{c_v \vrh^\triangleleft |D_t \vth |^2 }{2 |\xi_{\vt,h} |^2 } \right) \phi_h}
\br&
+ \intfaces{ |\avs{\vuh } \cdot \vn | \phi_h^{\rm down} \bigg( -c_v \vrhup\Echi{\vth^{\rm up}|\vth^{\rm down} } + \EB{\vrhup|\vrhdown} \bigg) }
\br &
+ h^\alpha \intfacesB{ \jump{ \vrh } \jump{ \nabla_{\vr}(-\vrh s_h)\phi_h } + \jump{ p_h } \jump{\nabla_p(-\vrh s_h) \phi_h } }.
\end{align*}
Further, reformulating the last two terms in the above equation as
\begin{align*}
&-c_v \vrhup\Echi{\vth^{\rm up}|\vth^{\rm down} } + \EB{\vrhup|\vrhdown} \bigg)
\\
& = \left( -\vrh s_h\right)^{\rm up} - \left( -\vrh s_h\right)^{\rm down} - \left(\nabla_{\vr}(-\vrh s_h ) \right)^{\rm down} \cdot \left( \vrh^{\rm up} -\vrh^{\rm down} \right)- \left(\nabla_{p}(-\vrh s_h ) \right)^{\rm down} \cdot \left( p_h^{\rm up} -p_h^{\rm down} \right)
\\
& = \frac12 \jump{ (\vrh, p_h) } \cdot \nabla_{(\vr,p)}^2 (-\vrh s_h)|_{\vv_1^*}\cdot \jump{ (\vrh, p_h) },
\quad\quad \vv_1^*\in \co{(\vrh^{\rm in}, p_h^{\rm in})}{(\vrh^{\rm out}, p_h^{\rm out})}~ \mbox{for any}~ \sigma \in \faces,
\end{align*}
and
\begin{align*}
& \jump{ \vrh } \jump{ \nabla_{\vr}(-\vrh s_h)\phi_h } + \jump{ p_h } \jump{\nabla_p(-\vrh s_h) \phi_h } \\
 & = \jump{\phi_h} \bigg( \avs{ \nabla_\vr(-\vrh s_h) } \jump{ \vrh} + \avs{\nabla_p(-\vrh s_h)} \jump{ p_h}\bigg)
 +\avs{\phi_h} \bigg( \jump{ \nabla_\vr(-\vrh s_h) } \jump{ \vrh}+ \jump{\nabla_p(-\vrh s_h)} \jump{ p_h}\bigg)
\\&
= \jump{\phi_h} \bigg( \avs{ \nabla_\vr(-\vrh s_h) } \jump{ \vrh}
	+ \avs{\nabla_p(-\vrh s_h)} \jump{ p_h}\bigg)
 + \avs{\phi_h} \jump{ (\vrh, p_h) } \cdot \nabla_{(\vr,p)}^2 (-\vrh s_h)|_{\vv_2^*}\cdot \jump{ (\vrh, p_h) }, \\
&\hspace{0.5\textwidth} \vv_2^*\in \co{(\vrh^{\rm in}, p_h^{\rm in})}{(\vrh^{\rm out}, p_h^{\rm out})}~ \mbox{for any}~ \sigma \in \faces
\end{align*}
 we obtain \eqref{eq_entropy_stability}. Finally, we finish the proof owing to the convexity of $-\vr s$ with respect to $(\vr,p)$.
\end{proof}

\subsection{Ballistic energy balance}
As the sign of the term $(\vth-\vthB)$ sitting in both the energy balance \eqref{energy_stability} and entropy balance \eqref{eq_entropy_stability} is undetermined, we can not get any a priori estimates directly from these two equations. Therefore, we need a new concept of stability -- the so-called ballistic energy balance.
\begin{Lemma}[Ballistic energy balance]\label{THBE}
Let $(\vrh ,\vuh ,\vth )$ be a numerical solution of the FV method \eqref{VFV} and $\phi_h \in W^{1,2}(0,T;Q_h)$. Then it holds
\begin{align}\label{BE0}
 &D_t \intTd{ \left(\frac{1}{2} \vrh |\vuh |^2 + c_v \vrh \vth - \vrh s_h \phi_h \right) }
 +\frac{1}{\penl} \intOsh{ |\vuh|^2}
 +\frac{1}{\penl} \intOsh{ \frac{(\vth-\vthB)^2 }{\vth} }
\br
& + \kappa \intTd{ \frac{\avs{ \phi_h} }{\vth \vthout } \; \abs{\Gradd \vth}^2 }
+ \intTd{ \frac{\phi_h}{\vth }\difuh }
+ D_s(\phi_h) + D_{\rm E}
\br
&=
-\intTd{ \vrh s_h ( D_t \phi_h + \vuh \cdot \Gradh \phi_h) }
+ \kappa \intTd{ \avs{ \frac1{\vth} } \Gradd \vth \cdot \Gradd \phi_h}
 + R_{B} (\phi_h) - R_{s}(\phi_h)
\end{align}
with $R_{B} = R_{B,1} +R_{B,2}$ given by
\begin{equation*}
\begin{aligned}
 R_{B,1} (\phi_h) &= \frac{1}{\penl}
 \intOsh { \frac{(\vth-\vthB) (\phi_h-\vthB) }{\vth} }
 + \TS \intTd{D_t (\vrh s_h) \cdot D_t \phi_h} ,
 \\
 R_{B,2} (\phi_h) &= \frac12 \intfaces{ |\avs{\vuh} \cdot \vc{n}| \jump{\vrh s_h } \jump{ \phi_h} }
 + \frac14 \intfaces{ \jump{\vuh} \cdot \vc{n} \jump{\vrh s_h } \jump{ \phi_h} } ,
 \end{aligned}
\end{equation*}
where $D_{\rm E}$ is given in \eqref{energy_dissipation}, $D_s(\phi_h)$ and $R_{s}(\phi_h)$ are given in \eqref{entropy_dissipation}.
\end{Lemma}

\begin{proof}
Firstly, using the algebraic equalities
\begin{align*}
& \phi_h D_t (\vrh s_h ) = D_t (\vrh s_h \phi_h) - (\vrh s_h - D_t(\vrh s_h) \TS ) D_t \phi_h
, \quad
\jump{\frac{\phi_h}{\vth}} = \jump{\phi_h} \avs{\frac{1}{\vth}} + \avs{\phi_h} \jump{ \frac{1}{\vth} }
\end{align*}
we reformulate the entropy balance \eqref{eq_entropy_stability} as
\begin{align*}
& D_t \intTd{ \vrh s_h \phi_h }
 + \frac{1}{\penl}\intOsh{ \frac{\vth-\vthB}{ \vth} \phi_h }
+\intfaces{ \frac{\kappa}{h} \avs{\phi_h} \jump{\vth } \jump{\frac{1}{\vth } } }
- \intTd{ \frac{\phi_h}{\vth } \difuh} 
\br
& = \intTd{ \vrh s_h D_t \phi_h }
- \TS \intTd{ D_t( \vrh s_h) D_t \phi_h }+
\intfaces{ \Up(\vrh s_h , \vuh ) \jump{\phi_h} }
\br
& \quad
- \intfaces{ \frac{\kappa}{h} \jump{\vth } \jump{\phi_h} \avs{ \frac{1}{\vth } } }
+ R_{s}(\phi_h) +D_s(\phi_h).
\end{align*}
Secondly, subtracting the above equation from the energy balance \eqref{energy_stability} we obtain
\begin{align*}\label{S6}
 &D_t \intTd{ \left(\frac{1}{2} \vrh |\vuh |^2 + c_v \vrh \vth - \vrh s_h \phi_h \right) }
 +\frac{1}{\penl} \intOsh{ |\vuh|^2}
 +\frac{1}{\penl} \intOsh{ \frac{(\vth-\vthB)^2 }{\vth} }
\\ & \quad
- \intfaces{ \frac{\kappa}{h} \avs{\phi_h} \jump{\vth } \jump{\frac{1}{\vth } } }
+ \intTd{ \frac{\phi_h}{\vth} \difuh }
+D_s(\phi_h) + D_{\rm E}
\\&
= -\intTd{ \vrh s_h D_t \phi_h }
- \intfaces{ \Up(\vrh s_h , \vuh ) \jump{\phi_h} }
+ \intfaces{ \frac{\kappa}{h} \jump{\vth } \jump{\phi_h} \avs{ \frac{1}{\vth } } }
\\&\quad
 +\frac{1}{\penl} \intOsh{ \frac{(\vth-\vthB) (\phi_h-\vthB) }{\vth} } + \TS \intTd{ D_t( \vrh s_h) D_t \phi_h } - R_{s}(\phi_h) .
\end{align*}
Reformulating the above formula with the following equalities
\begin{align*}
&\intfaces{ \Up(\vrh s_h , \vuh ) \jump{\phi_h} } = \intTd{\vrh s_h \vuh \cdot \Gradh \phi_h } - R_{B,2}(\phi_h),
\\& \intfaces{ \frac{\kappa}{h} \avs{\phi_h} \jump{\vth } \jump{\frac{1}{\vth } } } = -\kappa \intTd{ \frac{\avs{\phi_h}}{\vth \vthout } \; \abs{\Gradd \vth}^2 },
\\
&\intfaces{ \frac{\kappa}{h} \jump{\vth } \jump{\phi_h} \avs{ \frac{1}{\vth } } } = \kappa \intTd{ \avs{\frac1{\vth} } \; \Gradd \vth \cdot \Gradd \phi_h}
\end{align*}
we finish the proof.
\end{proof}

\subsection{Uniform bounds}
Now we are ready to derive the following a priori bounds. We start by showing a useful ``equality".
\begin{Lemma}
Let $\vuh \in Q_h^d, \, \mu > 0, \, \lambda \geq 0$. Then it holds
\begin{align}
\norm{\bS_h}_{L^2(\tor)} \approx \norm{\Gradh \vuh}_{L^2(\tor)} \approx \norm{\Dhuh}_{L^2(\tor)}.
\end{align}

\end{Lemma}
\begin{proof}
It is easy to check
\begin{align*}
& \bS_h : \Divh \vuh \I = (2\mu + d\lambda) \abs{\Divh \vuh}^2,
\quad
\intTd{\Gradh \vuh : \Gradh^T \vuh } = \norm{\Divh \vuh}_{L^2(\tor)}^2,
\br & \intTdB{\mu\abs{\Gradh \vuh}^2 + (\mu+\lambda) \abs{\Divh \vuh}^2}
= \intTd{\bS_h : \Gradh \vuh} = \intTd{\bS_h : \Dhuh }
\br &
= \intTdB{2\mu \abs{\bD_h \vuh}^2 + \lambda \abs{\Divh \vuh}^2},
\br
& \norm{\Divh \vuh}_{L^2(\tor)}^2 = \norm{\mbox{tr}(\Gradh \vuh)}_{L^2(\tor)}^2 = \norm{\mbox{tr}(\Dhuh)}_{L^2(\tor)}^2 \aleq \min\left( \norm{\Gradh \vuh}_{L^2(\tor)}^2 ,\norm{\Dhuh}_{L^2(\tor)}^2 \right).
\end{align*}
Consequently, we have
\begin{align*}
\norm{\Dhuh}_{L^2(\tor)} \aleq \norm{\Gradh \vuh}_{L^2(\tor)} \aleq \intTd{\bS_h : \Gradh \vuh} \aleq \norm{\bS_h}_{L^2(\tor)} \aleq \norm{\Dhuh}_{L^2(\tor)}
\end{align*}
and complete the proof.
\end{proof}

\begin{Lemma}[Uniform bounds]\label{lm_ub}
Let $(\vrh ,\vuh, \vth)$ be a numerical solution of the FV method \eqref{VFV} with $(\TS, h,\penl) \in (0,1)^3$ and $\alpha > -1$.
Assume that the numerical density and temperature are uniformly bounded from below and above
\begin{equation}\label{HP}
 0< \Un{\vr} \leq \vrh \leq \Ov{\vr}, \quad
 0< \Un{\vt} \leq \vth \leq \Ov{\vt} \ \mbox{ uniformly for }\ \TS, h, \penl \to 0.
\end{equation}
Additionally, we shall choose $\Un{\vt}, \Ov{\vt}$ satisfying $\Un{\vt} \leq \vtB \leq \Ov{\vt}$.
\vspace{0.1cm}

Then the following hold
\begin{subequations}\label{ap}
\begin{align}\label{ap1}
& \norm{p_h}_{L^\infty((0,T)\times(\tor))} + \norm{\vuh }_{L^\infty (0,T; L^{2}(\tor)) } + \norm{\Gradd \vth}_{L^2((0,T)\times\tor)} + \norm{\Dhuh}_{L^2((0,T)\times\tor)} \leq C ,
\end{align}
\begin{align}\label{ap2}
&(\TS)^{1/2}\left( \norm{D_t \vrh }_{L^2((0,T)\times\tor)} +\norm{D_t \vth }_{L^2((0,T)\times\tor)} +\norm{D_t \vuh }_{L^2((0,T)\times\tor)} \right) \leq C ,
\end{align}
\begin{align} \label{ap3}
\int_0^{\tau}\intfaces{ \left( h^\alpha + \abs{ \avs{\vuh}\cdot \vc{n} } \right) \left( \jump{\vrh}^2 + \jump{p_h}^2 + \abs{\jump{ \vuh }}^2 \right) }\dt \leq C,
\end{align}
\begin{equation}\label{ap4}
\frac{1}{\penl} \intTauOsh{\left( |\vuh|^2+ (\vth-\vthB)^2 \right)} \leq C.
\end{equation}
\end{subequations}
The constant $C$ depends on $\|\vtB\|_{W^{1,\infty}((0,T) \times \tor)}$ and $\Un{\vr}, \Ov{\vr}, \Un{\vt}, \Ov{\vt}$, but it is independent of the discretization parameters $(h, \TS)$ as well as the penalty parameter $\penl$.
\end{Lemma}

\begin{proof}[Proof of Lemma \ref{lm_ub}]
 We take $\phi_h=\vthB (:=\Pim \vtB)$ in the ballistic energy balance \eqref{BE0} and aim for an inequality that allows us to apply Gronwall's lemma. Before going to the details let us recall assumption \eqref{HP}, which implies
\begin{equation}\label{HP-1}
\Un{s} \leq s \leq \Ov{s} \quad  \mbox{ uniformly for }\ \TS, h, \penl \to 0.
\end{equation}
Now, applying \eqref{HP}, \eqref{HP-1} and \eqref{S6} we have the following estimates for the left-hand-side terms of the ballistic energy balance \eqref{BE0}
\begin{align}
& \frac{1}{\penl} \intOsh{ \frac{(\vth-\vthB)^2 }{\vth} } \ageq \frac{1}{\penl} \intOsh{ (\vth-\vthB)^2 },
\quad
 \kappa \intTd{ \frac{\avs\vthB}{\vth \vthout } \; \abs{\Gradd \vth}^2 } \ageq \norm{ \Gradd \vth}_{L^2(\tor)}^2,
\br &
\intTd{ \frac\vthB{\vth} \difuh } \ageq \norm{\Dhuh}_{L^2(\tor)}^2,
\quad
 D_{s,1}(\vthB) \ageq \ \TS \bigg(\norm{D_t \vrh }_{L^2(\tor)}^2 + \norm{D_t \vth }_{L^2(\tor)}^2 \bigg),
\br & \label{ds}
D_{s,2}(\vthB) \ageq \ \intfaces{ \abs{ \avs{\vuh}\cdot \vc{n} } \left( \jump{\vrh}^2 + \jump{p_h}^2\right) },
\quad D_{s,3}(\vthB) \ageq h^\alpha 
 \intfacesB{ \jump{\vrh}^2 + \jump{p_h}^2 },
\\ &
D_{\rm E} \ageq \TS \norm{D_t \vuh }_{L^2(\tor)}^2 + \intfaces{ \left( h^\alpha + |\avs{\vuh } \cdot \vc{n} | \right) \abs{\jump{\vuh}}^2 } \nonumber.
\end{align}
Note that in the estimates of $D_{s,2}, D_{s,3}$ we have used the fact that the Hessian matrix $\nabla_{(\vr,p)}^2 (-\vr s)$ is bounded from below by a positive constant under assumption \eqref{HP}.

 Next, applying the interpolation estimates and Young's inequality we have the following estimates for the right-hand-side terms of the ballistic energy balance \eqref{BE0}
\begin{align*}
 \Bigabs{\intTd{ \vrh s_h D_t \vthB }} & \aleq \norm{\vtB}_{W^{1,\infty}(0,T;L^1(\tor))} \aleq 1,
\\
\Bigabs{ \intTd{\vrh s_h \vuh \cdot \Gradh\vthB} } &\aleq \norm{\vtB}_{L^{\infty}(0,T;W^{1,\infty}(\tor))} \intTd{ \abs{\vrh \vuh}}
 \aleq 1 + \intTd{ \frac12 \vrh \abs{\vuh}^2},
\\
\Bigabs{ \intTd{ \avs{\frac1{\vth}} \; \Gradd \vth \cdot \Gradd \vthB}} & \aleq \norm{\vtB}_{L^{\infty}(0,T;W^{1,\infty}(\tor))} \intTd{ \abs{\Gradd \vth}}
 \\
 &\aleq \intTdB{ \delta \abs{\Gradd \vth}^2 + \frac1{\delta}}
 \aleq \frac{1}{\delta} + \delta \norm{\Gradd \vth}_{L^2(\tor)}^2,
\end{align*}
for a suitable $\delta \in \R$. Similarly,
\begin{align}
&\Bigabs{\TS \intTd{D_t (\vrh s_h) \; D_t \vthB} } \aleq \norm{\vtB}_{W^{1,\infty}(0,T;L^{\infty}(\tor))} \TS \intTdB{ \abs{D_t \vrh} + \abs{D_t \vth} }
\br
&\hspace{1.8cm} \aleq \frac{\TS}{\delta}+ \TS \delta \bigg(\norm{D_t \vrh }_{L^2(\tor)}^2 + \norm{D_t \vth }_{L^2(\tor)}^2 \bigg),
 \br
 &\Bigabs{R_{s}(\vthB)} \aleq \norm{\vtB}_{L^{\infty}(0,T;W^{1,\infty}(\tor))} h^{\alpha+1} \intfaces{ \bigg( \abs{ \jump{ \vrh}}
	+ \abs{ \jump{ p_h}} \bigg)}
\br
&\hspace{1.8cm}\aleq h^{\alpha+1} \intfaces{ \left[\frac{h}{\delta} + \frac{\delta}{h} \bigg( \jump{ \vrh}^2	+ \jump{ p_h}^2 \bigg) \right]}
 \aleq \frac1{\delta} h^{\alpha+1} + \delta h^\alpha \intfacesB{ \jump{\vrh }^2 + \jump{p_h }^2 } ,
\label{rs}
\\& \Bigabs{R_{B,2}(\vthB)} \aleq \norm{\vtB}_{L^{\infty}(0,T;W^{1,\infty}(\tor))} h \intfaces{ \abs{\vuh}} \aleq \intTd{ \abs{ \vuh} }
 \aleq 1 + \intTd{\frac12 \vrh \abs{\vuh}^2}. \nonumber
\end{align}
Collecting the above estimates we derive the ballistic energy inequality
\begin{align}\label{BEI}
 &D_t \intTd{ \left(\frac{1}{2} \vrh |\vuh |^2 + c_v \vrh \vth - \vrh s_h \vthB \right) }
 +\frac{1}{\penl} \intOsh{ |\vuh|^2}
 +\frac{1}{\penl} \intOsh{ (\vth-\vthB)^2 }
\br
&\hspace{1cm} + \TS\bigg(\norm{D_t \vrh }_{L^2(\tor)}^2 + \norm{D_t \vth }_{L^2(\tor)}^2 + \norm{D_t \vuh }_{L^2(\tor)}^2\bigg) + \norm{\Gradd \vth}_{L^2(\tor)}^2 + \norm{\Dhuh}_{L^2(\tor)}^2
\br
&\hspace{1cm} +\intfaces{ \left( h^\alpha + \abs{ \avs{\vuh}\cdot \vc{n} } \right) \left( \jump{\vrh}^2 + \jump{p_h}^2 + \abs{\jump{ \vuh }}^2 \right) }
\br
\aleq \ &1+ \frac{1}{\delta} + \frac{1}{\delta}h^{\alpha+1}+ \intTd{ \frac12 \vrh \abs{\vuh}^2} + \delta \TS \bigg(\norm{D_t \vrh }_{L^2(\tor)}^2 + \norm{D_t \vth }_{L^2(\tor)}^2 \bigg)
\br
&+ \delta \ \left( \norm{\Gradd \vth}_{L^2(\tor)}^2 + h^\alpha \intfacesB{ \jump{\vrh }^2 + \jump{p_h }^2 } \right).
\end{align}
Hence, choosing $\alpha>-1$, $\delta \in (0,1)$ and applying Gronwall's lemma we obtain from $\Abs{\intTd{(\vrh s_h \vthB)(\tau,\cdot)}} \aleq 1$ and above ballistic energy inequality \eqref{BEI} that
\begin{align*}
 &\intTd{ \left(\frac{1}{2} \vrh |\vuh |^2 + c_v \vrh \vth \right)({\tau},\cdot) }
 +\frac{1}{\penl} \intTauOsh{\left( |\vuh|^2 + (\vth-\vthB)^2 \right)}
\br
&\hspace{1cm} + \TS\intTauTdB{ \abs{ D_t \vrh }^2 +\abs{ D_t \vth }^2+\abs{ D_t \vuh }^2} + \intTauTdB{ \abs{\Gradd \vth}^2 + | \Dhuh|^2 }
\br
&\hspace{1cm}+\int_0^{\tau}\intfaces{ \left( h^\alpha + \abs{ \avs{\vuh}\cdot \vc{n} } \right) \left( \jump{\vrh}^2 + \jump{p_h}^2 + \abs{\jump{ \vuh }}^2 \right) }\dt
\br
& \aleq 1+ \intTd{ \left(\frac{1}{2} \vrh^0 |\vuh^0 |^2 + c_v \vrh^0 \vth^0 \right)} \aleq 1
\hspace{4cm}
\mbox{with}~ \tau = t^k, k = 1,\dots,N_T,
\end{align*}
which gives the a priori estimates \eqref{ap} and concludes the proof.
\end{proof}

\section{Consistency}\label{sec_con}
Having shown the stability of our FV method \eqref{VFV} we are now ready to derive its consistency. The proof is similar to those presented in literature results, e.g. \cite{FeLMMiSh}, admitting differences in technical details. Thus, we postpone them to Appendix \ref{app-cf} for better readability.

\begin{Lemma}[Consistency of the continuity, momentum and entropy equations]\label{lem_C1}
Let $(\vrh, \vuh, \vth)$ be a numerical solution obtained by the FV scheme \eqref{VFV} with $(\TS, h,\penl) \in (0,1)^3$ and $-1 < \alpha <1$.
Let assumption \eqref{HP} hold.

Then for any $\tau \in [0,T]$ we have
\begin{subequations}\label{eq_C1_D}
\begin{equation} \label{cP1_D}
\left[ \intTd{ \vrh\phi } \right]_{t=0}^{t=\tau}=
 \intTauTdB{ \vrh \partial_t \phi + \vrh \vuh \cdot \Grad \phi } + e_\vr(\phi)
\end{equation}
for any $\phi \in C^2([0,T] \times \tor)$;
\begin{align} \label{cP2_D}
\left[ \intOf{ \vrh \vuh \cdot \bfphi } \right]_{t=0}^{t=\tau} = &
\intTauOfB{ \vrh \vuh \cdot \partial_t \bfphi + \vrh \vuh \otimes \vuh : \Grad \bfphi } \br
&- \intTauOf{ ( \bS_h -p_h \I) : \Grad \bfphi }
+ e_{\vm}(\bfphi)
 \end{align}
for any $\bfphi \in C^2_c([0,T] \times \Of; \R^d)$;
\begin{align}\label{cP3_D}
\left[ \intOf{ \vrh s_h \phi } \right]_{t=0}^{t=\tau} & \geq
 \int_0^\tau\intOf{ \vrh s_h (\pd_t\phi + \vuh \cdot \Grad \phi ) } \dt - \int_0^\tau\intOf{ \frac{\kappa}{\vth} \Gradd \vth \cdot \Grad \phi } \dt
\br
&+ \int_0^\tau\intOf{ \frac{ \kappa \phi}{ \vthout \vth} \abs{\Gradd \vth }^2 } \dt
+ \int_0^\tau\intOf{\difuh \frac{ \phi}{\vth} } \dt
 + e_{s}(\phi)
\end{align}
for any $\phi \in C^2_c([0,T] \times \Of)$, $\phi \geq 0$.

The consistency errors satisfy
\begin{align*}
 \Bigabs{e_\vr(\phi) } & \leq C_\varrho\left( \TS +h+h^{1-\alpha} + h^{\alpha+1}\right),
 \\
 \Bigabs{ e_{\vm}(\bfphi) }& \leq C_{\vm}\left(\TS+h+h^{(1-\alpha)/2}+h^{\alpha+1} +h^{3/2} \penl^{-1/2} \right), 
 \\
\Bigabs{ e_{s}(\phi) } &\leq C_s\left( \TS +h+h^{1-\alpha} + h^{(1+\alpha)/2} + h^{3/2} \penl^{-1/2}\right),
\end{align*}
\end{subequations}
where the constants $C_\vr,$ $C_{\vm},$ $C_{s}$ are independent of parameters $\TS,h,\penl$.
\end{Lemma}
\begin{Lemma}[Consistency of the ballistic energy inequality]\label{lem_C4}
Let $(\vrh, \vuh, \vth)$ be a numerical solution obtained by the finite volume scheme \eqref{VFV} with $(\TS, h,\penl) \in (0,1)^3$ and $-1 < \alpha <1$.
Let assumption \eqref{HP} hold.

Then, for any $\tau \in [0,T]$ and any $\hvt \in W^{2,\infty}((0,T) \times \tor),\ \inf \hvt > 0,\ \hvt|_{\Os} = \vtB$ it holds that
\begin{subequations}\label{eq_C4}
\begin{align}\label{cP4}
 &\left[ \intTd{ \left(\frac{1}{2} \vrh |\vuh |^2 + c_v \vrh \vth - \vrh s_h \hvt \right) (t, \cdot)} \right]_{t=0}^\tau
 + \intTauTdB{ \frac{ \kappa\hvt}{\vth \vthout } \; \abs{\Gradd \vth}^2 + \frac{\hvt}{\vth }\difuh}
\br
 & \quad +\frac{1}{\penl} \intTauOshB{ |\vuh|^2 + \frac{(\vth-\vthB)^2 }{\vth} }
\br
 & \leq - \intTauTdB{ \vrh s_h \partial_t \hvt + \vrh s_h \vuh \cdot \Grad \hvt - \frac{\kappa}{\vth} \; \Gradd \vth \cdot \Grad \hvt} + e_{B}(\hvt) + C_B' h^{1+\alpha},
\end{align}
where
 \begin{align}\label{ec4}
 \abs{e_{B}(\hvt) } \leq C_{B}\left( (\TS)^{1/2} +h+h^{1-\alpha} + h^{3/2} \penl^{-1/2}\right)
 \end{align}
 and constants $C_B, C_B'$ are independent of parameters $\TS,h,\penl$.
\end{subequations}
\end{Lemma}

\section{Convergence}\label{sec_convergence}
We are now ready to study the convergence of the penalty FV method \eqref{VFV} to a DMV solution. Consequently, we prove the existence of global DMV solution for the NSF system \eqref{pde} with the Dirichlet boundary conditions. In addition, we show the strong convergence to the strong solution as long as the latter exists.

\subsection{DMV solution}
To begin we introduce the definition of the DMV solution proposed in the work of Chaudhuri \cite{Chaudhuri:2022}.
\begin{Definition}[DMV solution of the Dirichlet problem \cite{Chaudhuri:2022}] \label{DMV}
A parametrized family of probability measures $\{\mathcal{V}_{t,x}\}_{(t,x)\in (0,T)\times\Of}$ is the DMV solution to the Navier--Stokes--Fourier system \eqref{pde}--\eqref{ic} with the initial condition $\{\mathcal{V}_{0,x}\}_{x\in\Of}$ if the following hold:
\begin{itemize}
\item {\bf Mapping}
\begin{equation*}
\mathcal{V}_{t,x} : (t,x) \in (0,T)\times \Of \longmapsto \mathcal{P}(\mathcal{F}) \quad \mbox{is weakly-}(^*) measurable
\end{equation*}
with $\mathcal{P}$ being the space of probability measures defined on the phase space
\begin{equation*}
\mathcal{F} = \left\{ \vr, \vu, \vt, \bD_{\vu}, \bD_{\vt} \, \bigg| \, \vr \geq 0, \vt \geq 0, \vu \in \R^d, \bD_{\vu} \in \R^{d\times d}_{sym}, \bD_{\vt} \in \R^d \right\}.
\end{equation*}

\item {\bf Compatibility condition}
\begin{align}\label{dmv_cc}
&-\intTOf{ \mytangle{\vu} \cdot \Div \mathbb{T} } = \intTOf{\mytangle{ \bD_{\vu}} : \mathbb{T} } ,
\br
&-\intTOf{\mytangle{ \vt - \hvt} \cdot \Div \bfphi } = \intTOf{\mytangle{ \bD_{\vt} - \Grad \hvt} \cdot \bfphi }
\end{align}
for any $ \mathbb{T} \in C^{1}([0,T]\times \Ov{\Of};\R^{d\times d}_{sym})$, $\bfphi \in C^{1}([0,T]\times \Ov{\Of};\R^{d})$ and $\hvt \in C^1([0,T] \times \Ov{\Of}),\ \hvt|_{\partial \Of} = \vtB$.

\item {\bf Continuity equation}
		\begin{equation} \label{dmv_d}
		\intOf{\myTauangle{ \vr}\, \phi(\tau,\cdot) } -\intOf{\myZangle{ \vr}\, \phi(0,\cdot) } = \intTauOfB{ \mytangle{\vr} \partial_t \phi + \mytangle{\vr \vu} \cdot \Grad \phi }
		\end{equation}
	for a.a. $\tau \in (0,T)$ and any $\phi \in C^{1}([0,T] \times \Ov{\Of})$.
			
\item {\bf Momentum equation}
\begin{align}\label{dmv_m}
	&\intOf{ \myTauangle{\vr \vu} \cdot \bfphi(\tau,\cdot) } - \intOf{ \myZangle{\vr \vu} \cdot \bfphi(0,\cdot) }
	\br
	& = \intTauOfB{ \mytangle{\vr \vu} \cdot \partial_t \bfphi + \mytangle{\vr \vu \otimes \vu} : \Grad \bfphi + \mytangle{p(\vr,\vt)} \, \Div \bfphi }
	\br
	& - \intTauOf{ \mytangle{\bS(\bD_{\vu}) }: \Grad \bfphi }+ \int_0^{\tau} \int_{\Of}\Grad \bfphi : d\mathfrak{R}(t) \dt
\end{align}	
for a.a. $\tau \in (0,T)$ and any $\bfphi \in C_c^1([0,T] \times \Of; \R^d)$
with the Reynolds defect measure
	 \begin{equation*}
	 \mathfrak{R} \in L^{\infty}(0,T; \mathcal{M}(\Of; \mathbb{R}^{d\times d}_{\mbox{sym}})).
	 \end{equation*}

	\item {\bf Entropy inequality}	
	\begin{align}\label{dmv_s}
	&\intOf{ \myTauangle{\vr s(\vr, \vt)} \cdot \phi(\tau,\cdot) }- \intOf{ \myZangle{\vr s(\vr, \vt)} \cdot \phi(0,\cdot) }
	\br
	 &\geq \intTauOfB{ \mytangle{\vr s(\vr, \vt)} \partial_t \phi + \mytangle{\vr s(\vr, \vt) \vu- \frac{\kappa \bD_{\vt}}{\vt} } \cdot \Grad \phi } \br
	 &+ \intTauOf{ \mytangle{ \frac{1}{\vt}\left(\bS(\bD_{\vu}): \bD_{\vu} + \frac{\kappa\abs{\bD_{\vt}}^2}{\vt} \right)}\phi }
\end{align}
for a.a. $\tau \in (0,T)$ and any $\phi \in C_c^1([0,T] \times \Of), \, \phi \geq 0$.
	
\item {\bf Ballistic energy inequality}:	
	For any
\begin{equation}\label{tvB}
\hvt \in C^1([0,T] \times \Ov{\Of}),\ \inf \hvt > 0,\ \hvt|_{\partial \Of} = \vtB,
\end{equation}
there exists a dissipation defect measure
	\begin{equation*}
		\mathfrak{B}_{\hvt} \in L^{\infty}(0,T; \mathcal{M}^+(\Ov{\Of}))
	\end{equation*}
such that
	\begin{align}\label{dmv_E}
		&\intOf{ \myTauangle{\frac{1}{2} \vr |\vu|^2 + \vr e(\vr,\vt) - \hvt(\tau,\cdot) \vr s(\vr, \vt) } } + \int_{\Ov{\Of}} d \mathfrak{B}_{\hvt}(\tau)
		 \br
		 &+ \int_{0}^{\tau} \intOf{ \mytangle{\frac{1}{\vt} \left( \bS(\bD_{\vu}): \bD_{\vu} +\frac{\kappa\abs{\bD_{\vt}}^2}{\vt}\right)} \hvt } \dt
		 \br
		 & \leq \intOf{ \myZangle{\frac{1}{2} \vr |\vu|^2 + \vr e(\vr,\vt) - \hvt(0,\cdot) \vr s(\vr, \vt) } }
		 \br
		 & - \int_0^\tau \intOfB{ \mytangle{\vr s(\vr, \vt)} \partial_t \hvt + \mytangle{\vr s(\vr, \vt) \vu } \cdot
					\Grad \hvt - \mytangle{\frac{\kappa\bD_{\vt}}{\vt}} \cdot \Grad \hvt } \dt
	\end{align}
and
	 \begin{equation}
	 	\Bigabs{ \int_{\Of}\Grad \bfphi : d\mathfrak{R}(\tau) } \aleq \norm{ \bfphi }_{C^1(\Of)} \int_{\Of} d \mathfrak{B}_{\hvt}(\tau) \quad \mbox{for any } \bfphi \in C^1(\Of).
	 \end{equation}
\end{itemize}
\end{Definition}

\subsection{Weak convergence}\label{sec-converge-of}
In this section we consider the limit process with $\TS, h, \penl \to 0$.
Thanks to uniform bounds \eqref{ap1} and the Fundamental Theorem on Young measure \cite{Ball:1989} we obtain that, up to a subsequence, the numerical solutions
$\{ \vrh, \vuh, \vth, \Dhuh, \Gradd \vth \}_{\TS, h, \penl \searrow 0}$ in the limit for $\TS, h, \penl \to 0$ generate a Young measure $\{\mathcal{V}_{t,x}\}_{(t,x)\in (0,T)\times\Of}$, whose support satisfies
\begin{equation*}
\mbox{supp}[\mathcal{V}_{t,x}] \subset \left\{ \vr, \vu, \vt, \bD_{\vu}, \bD_{\vt} \, \bigg| \, 0 < \Un{\vr}\leq \vr \leq \Ov{\vr}, \, 0 < \Un{\vt} \leq \vt \leq \Ov{\vt}, \, \vu \in \R^d, \bD_{\vu} \in \R^{d\times d}_{sym}, \bD_{\vt} \in \R^d \right\}
\end{equation*}
for a.a. $(t,x) \in (0,T)\times \Of$. The mapping $\mathcal{V}_{t,x} : (t,x) \in (0,T)\times \Of \longmapsto \mathcal{P}(\mathcal{F})$ is weakly-$(^*)$ measurable.

Furthermore, we know from a priori bound \eqref{ap4} that
\begin{align}\label{eq-convegence-os}
\vuh \to 0 \ \mbox{ strongly in}\ L^2((0,T)\times \Os; \R^d), \quad
\vth \to \vtB \ \mbox{ strongly in}\ L^2((0,T)\times \Os; \R^d).
 \end{align}
Combining with the density consistency \eqref{cP1_D} we obtain
\begin{align}\label{eq-convegence-os-1}
&\partial_t \vr = 0 \ \mbox{in}\ \mathcal{D}'((0,T) \times \Os ) ~
\ \mbox{yielding}
~ \ \mytangle{\vr}= \vr_{0}^s \ \ \mbox{in}\ \Os \quad \mbox{for all } t\in(0,T).
\end{align}
We let additionally
\begin{align}\label{HP-ic}
\vr_{0}^s = \begin{cases}
 0 & \ \mbox{if} \ \vtB = \vtB(t, x), \\
 \vr_{0}^s(x) > 0 & \ \mbox{if} \ \vtB = \vtB(x).
\end{cases}
\end{align}
Consequently,
\begin{itemize}
\item the compatibility condition \eqref{dmv_cc} can be derived as follows.
 Thanks to the uniform bounds \eqref{ap} and the projection errors \eqref{proj}, we obtain the compatibility condition on $\tor$:
\begin{equation}\label{compatibility-td}
\begin{aligned}
&-\intTTd{ \mytangle{\vu} \cdot \Div \mathbb{T} } = \intTTd{\mytangle{ \bD_{\vu}} : \mathbb{T} } ,
\\
&-\intTTd{\mytangle{ \vt - \hvt} \cdot \Div \bfphi } = \intTTd{\mytangle{ \bD_{\vt}-\Grad \hvt } \cdot \bfphi }
\end{aligned}
\end{equation}
for any $ \mathbb{T} \in C^{1}([0,T]\times \tor;\R^{d\times d}_{sym})$, $\bfphi \in C^{1}([0,T]\times \tor;\R^{d})$ and $\hvt \in C^1([0,T] \times \Ov{\Of}),\ \hvt|_{\Os} = \vtB$. Combining with \eqref{eq-convegence-os} we have
\begin{align*}
&\mytangle{ \bD_{\vu}} = 0, \quad \mytangle{ \bD_{\vt}} = \Grad \vtB \quad \mbox{in}\ \Os \quad \mbox{for all } t\in(0,T).
\end{align*}
Hence, combining \eqref{compatibility-td} with
\begin{align*}
&\intTOs{ \mytangle{\vu} \cdot \Div \mathbb{T} } = \intTOs{\mytangle{ \bD_{\vu}} : \mathbb{T} } = 0,
\br
& \intTOs{\mytangle{ \vt - \hvt} \cdot \Div \bfphi } = \intTOs{\mytangle{ \bD_{\vt} - \Grad \hvt} \cdot \bfphi } = 0
\end{align*}
we have the compatibility condition \eqref{dmv_cc}.

\item the continuity equation \eqref{dmv_d} follows from the density argument for the test functions and \eqref{HP-ic}, i.e.
\begin{align*}
& \lim_{\TS,h,\penl\to 0} \left( \left[ \intOs{ \vrh\phi } \right]_{t=0}^{t=\tau} -
 \intTauOs{ \vrh \partial_t \phi } \right) = 0,
\quad \lim_{\TS,h,\penl\to 0} \intTauOs{ \vrh \vuh \cdot \Grad \phi } = 0.
\end{align*}

\item the momentum equation \eqref{dmv_m} and entropy inequality \eqref{dmv_s} are directly obtained by applying the density argument for the test functions. We note that
\begin{align*}
& \mathfrak{R}(t) = \Ov{\vr \vu \otimes \vu} - \mytangle{\vr \vu \otimes \vu} \in L^{\infty}(0,T; \mathcal{M}(\Of; \mathbb{R}^{d\times d}_{\mbox{sym}})), \br
& {\mbox and} \quad \vrh \vuh \otimes \vuh \ \longrightarrow \ \Ov{\vr \vu \otimes \vu} \ \mbox{weakly in}\ L^{\infty}(0,T; \mathcal{M}(\Of; \mathbb{R}^{d\times d}_{\mbox{sym}}))
\mbox{ as } (h,\TS,\penl) \to 0.
\end{align*}

\item the ballistic energy inequality \eqref{dmv_E} can be derived as follows. Specifically, it is easy to check
\begin{align*}
& \frac{1}{\penl} \intTauOshB{ |\vuh|^2 + \frac{(\vth-\vthB)^2 }{\vth} } \geq 0, \quad \intTauOs{ \frac{\hvt}{\vth }\difuh} \geq 0,
\br
& \lim_{\TS,h,\penl\to 0} \intTauOs{ \frac{ \kappa\hvt}{\vth \vthout } \; \abs{\Gradd \vth}^2 } \geq \intTauOs{ \frac{ \kappa}{ \vtB } \; \abs{\Grad \vtB}^2 },
\br
&\lim_{\TS,h,\penl\to 0}\intTauOs{ \vrh s_h \partial_t \hvt } = 0,
\quad
\lim_{\TS,h,\penl\to 0} \intTauOs{ \vrh s_h \vuh \cdot \Grad \hvt } = 0,
\br
&\lim_{\TS,h,\penl\to 0} \intTauOs{ \frac{\kappa}{\vth} \; \Gradd \vth \cdot \Grad \hvt} = \intTauOs{ \frac{ \kappa}{\vtB } \; \abs{\Grad \vtB}^2 },
\br
&\lim_{\TS,h,\penl\to 0} \intOsB{ \frac{1}{2} \vrh^0 |\vuh^0 |^2 + c_v \vrh^0 \vth^0 - \vrh^0 s_h^0 \hvt(0,\cdot) }
= \intOsB{ c_v \vr_0^s \vtB - \vr_0^s s(\vr_0^s,\vtB(0,\cdot) ) \vtB(0,\cdot) },
\br
&\lim_{\TS,h,\penl\to 0} \intOs{ \left( \frac{1}{2} \vrh |\vuh |^2 + c_v \vrh \vth - \vrh s_h \hvt \right)(\tau,\cdot) }
= \intOsB{ c_v \vr_0^s \vtB - \vr_0^s s(\vr_0^s,\vtB) \vtB}
\end{align*}
and
\begin{align*}
&\lim_{\TS,h,\penl\to 0} \intTauOfB{ \frac{ \kappa\hvt}{\vth \vthout } \; \abs{\Gradd \vth}^2 + \frac{\hvt}{\vth }\difuh}
\br
&\quad\geq \intTauOf{ \mytangle{\frac{1}{\vt} \left( \bS(\bD_{\vu}): \bD_{\vu} +\frac{\kappa\abs{\bD_{\vt}}^2}{\vt}\right)} \hvt },
\br
&\lim_{\TS,h,\penl\to 0} \intTauOfB{ \vrh s_h \partial_t \hvt + \vrh s_h \vuh \cdot \Grad \hvt - \frac{\kappa}{\vth} \; \Gradd \vth \cdot \Grad \hvt}
\br
&\quad= \int_0^\tau \intOfB{ \mytangle{\vr s(\vr, \vt)} \partial_t \hvt + \mytangle{\vr s(\vr, \vt) \vu } \cdot
					\Grad \hvt - \mytangle{\frac{\kappa\bD_{\vt}}{\vt}} \cdot \Grad \hvt } \dt,
\br
&\lim_{\TS,h,\penl\to 0} \intOfB{ \frac{1}{2} \vrh^0 |\vuh^0 |^2 + c_v \vrh^0 \vth^0 - \vrh^0 s_h^0 \hvt(0,\cdot) }
= \intOf{ \myZangle{\frac{1}{2} \vr |\vu|^2 + \vr e(\vr,\vt) - \hvt \vr s(\vr, \vt) } } ,
\br
&\lim_{\TS,h,\penl\to 0} \intOf{ \left( \frac{1}{2} \vrh |\vuh |^2 + c_v \vrh \vth - \vrh s_h \hvt \right)(\tau,\cdot) }
\br
&\quad= \int_{\Ov{\Of}} d \mathfrak{B}_{\hvt}(\tau) + \intOf{ \myTauangle{\frac{1}{2} \vr |\vu|^2 + \vr e(\vr,\vt) - \hvt \vr s(\vr, \vt) } }
\end{align*}
with
\begin{align*}
& \mathfrak{B}_{\hvt}(\tau) := \frac{1}{2} \Ov{\vr |\vu |^2} - \myTauangle{\frac{1}{2} \vr |\vu|^2 } \in L^{\infty}(0,T; \mathcal{M}^+(\Ov{\Of}))
\br
&\mbox{and} \quad \vrh |\vuh |^2 \ \longrightarrow \ \Ov{\vr |\vu |^2} \ \mbox{weakly in}\ L^{\infty}(0,T; \mathcal{M}^+(\Ov{\Of}))
\mbox{ as } (\TS, h, \penl) \to 0.
\end{align*}
Collecting the above formulae and \eqref{HP-ic} we obtain the ballistic energy inequality \eqref{dmv_E}.
\end{itemize}
Summarizing the obtained results we have derived the following convergence theorem.
\begin{Theorem}\rm\label{THM2}
(\textbf{Weak convergence})
Let $(\vrh, \vuh, \vth)$ be a numerical solution obtained by the finite volume scheme \eqref{VFV} with $(\TS, h,\penl) \in (0,1)^3$ and $-1 < \alpha <1$.
Let the numerical density $\vrh$ and temperature $\vth$ be uniformly bounded, i.e. there exist $\Un{\vr}, \Ov{\vr}, \Uvt, \Ovt $ such that
\begin{equation*}
		 0< \Un{\vr} \leq \vrh \leq \Ov{\vr} ,\ 0< \Uvt \leq \vth \leq \Ovt \ \mbox{ uniformly for }\ \TS, h, \penl \to 0.
\end{equation*}

Then, up to a subsequence, the family
$\{ \vrh, \vuh, \vth, \Dhuh, \Gradd \vth \}_{\TS, h, \penl \searrow 0}$ generates a Young measure $\{\mathcal{V}_{t,x}\}_{(t,x)\in (0,T)\times\Of}$ in the limit as $\TS, h, \penl \to 0$ such that $h^3/\penl \to 0$. Here $\{\mathcal{V}_{t,x}\}$ is a DMV solution of the Dirichlet problem \eqref{pde}--\eqref{ic} in the sense of Definition \ref{DMV}.	
\end{Theorem}

\subsection{Strong convergence}\label{convergence-strong}
In this section we study the convergence of the numerical approximations towards a strong solution. From now on we focus on the time independent boundary condition, i.e. $\vtB = \vtB(x)$, for simplicity.
To begin, we first introduce the class of a strong solution of the NSF system \eqref{pde}, see \cite{Valli:1982a,Valli:1982b} and \cite{FLM_NSF}.
\begin{Definition}[{\bf Strong solution}]\label{def-ss}
 Let $\Of \subset \R^d, d=2,3,$ be a bounded domain with a smooth boundary $\partial \Of \in C^k$.
Let the initial/boundary data $(\vr_0, \vu_0,\vt_0,\vtB) $ satisfy
\begin{equation}\label{ic-compat}
\begin{aligned}
&(\vr_0, \vu_0,\vt_0) \in W^{k,2}(\R^{d+2}), \quad \quad \vtB|_{\partial \Of} \in W^{k+1/2,2}(\partial \Of),
\\
& \Grad p(\vr_0, \vtB) = \Div \mathbb{S}(\bD(\vu_0)), \quad 	
 \Div \vc{q} (\vtB) = \mathbb{S} (\bD(\vu_0)) : \bD(\vu_0) - p(\vr_0, \vtB) \Div \vu_0 \ \ \mbox{for}\
	x \in \partial \Of.
\end{aligned}
\end{equation}
We say that $(\vr, \vu, \vt)$ is the strong solution of the Navier--Stokes-Fourier problem \eqref{pde} if
\begin{equation}\label{STC}
\begin{aligned}
&\vr \in C(0,T; W^{k,2}(\Of)) \cap C^1(0,T; W^{k-1,2}(\Of)),
\\&
\vu \in C(0,T; W^{k,2}(\Of;\R^d)) \cap L^2(0,T; W^{k+1,2}(\Of;\R^d))\cap W^{1,2}(0,T; W^{k-1,2}(\Of;\R^d)),
\\&
\vt \in C(0,T; W^{k,2}(\Of)) \cap L^2(0,T; W^{k+1,2}(\Of)) \cap W^{1,2}(0,T; W^{k-1,2}(\Of)), \quad \quad k \geq 4
\end{aligned}
\end{equation}
and equations \eqref{pde} are satisfied pointwise.
\end{Definition}

\begin{Remark}\label{rmk-strong}
Applying the Sobolev embedding theorem we have $(\vr,\vu,\vt)\in L^{\infty}(0,T; W^{2,\infty}(\Of;\R^{d+2}))$.
Furthermore, following the proof in \cite[Lemma 2.1]{Valli-Zajaczkowski:1986} we obtain $\partial_t^2(\vr,\vu,\vt)\in L^2(0,T;W^{k-3,2}(\Of;\R^{d+2}))$ and $\partial_t(\vr,\vu,\vt)\in L^{\infty}(0,T;W^{k-2,2}(\Of;\R^{d+2})) \hookrightarrow \hookrightarrow L^{\infty}((0,T)\times \Of;\R^{d+2})$.
\end{Remark}

\begin{Theorem}\rm\label{THM3}
(\textbf{Strong convergence})
 Let $\Of \subset \R^d, d=2,3,$ be a bounded domain with a smooth boundary $\partial \Of \in C^k$, $k\geq 4$.
Let the initial/boundary data $(\vr_0, \vu_0,\vt_0,\vtB) $ satisfy \eqref{ic-compat}.
Let $(\vrh, \vuh, \vth)$ be a numerical solution obtained by the finite volume scheme \eqref{VFV} with $(\TS,h,\penl) \in (0,1)^3$, $-1 < \alpha <1$.

Under the assumption
\begin{align*}
 0< \Un{\vr} \leq \vrh \leq \Ov{\vr}, \quad
 0< \Un{\vt} \leq \vth \leq \Ov{\vt}, \quad
  \abs{\vuh} \leq \Ov{u} \quad \mbox{ uniformly for }\ \TS,h,\penl \to 0
\end{align*}
we have that the Dirichlet problem \eqref{pde} admits a strong solution $(\vr, \vu,\vt)$ in the sense of Definition \ref{def-ss}.
Moreover, the FV solutions $\{ \vrh, \vuh, \vth \}_{h,\penl\searrow 0}$ converge strongly in the following sense
\begin{align*}
		\vrh &\longrightarrow \vr \quad \quad \mbox{strongly in}\ L^p((0,T)\times\Of), \br
		\vuh &\longrightarrow \vu \quad \quad \mbox{strongly in}\ L^{p}((0,T)\times\Of; \R^d), \br
		\vth &\longrightarrow \vt \quad \quad \mbox{strongly in}\ L^p((0,T)\times\Of)
  \qquad \mbox{ as } \TS,h,\penl \to 0
\end{align*}
for any $p \in [1,\infty)$.
\end{Theorem}

\begin{proof}
According to Theorem \ref{THM2} we have the global existence of a DMV solution. Then, the global existence of a strong solution $(\vr, \vu,\vt)$ directly follows from the local existence of a strong solution proven in \cite{Valli:1982b,Valli:1982a,Valli-Zajaczkowski:1986} and the conditional regularity shown in \cite{BFM}. Further, by the virtue of the weak-strong uniqueness \cite{Chaudhuri:2022} we obtain the strong convergence of our numerical solutions.
\end{proof}
\medskip

Having shown the strong convergence of the penalization FV approximations $(\vrh,\vuh,\vth)_{\TS,h,\penl \searrow 0}$, the next interesting question is to study the convergence rate.
To this end, we first define a suitable extension of the strong solution on $\tor$.
\begin{Definition}[Extension of the strong solution]
\label{def_ES}
Let $(\vr,\vu,\vt)$ be the strong solution of Dirichlet NSF system \eqref{pde} in the sense of Definition \ref{def-ss}. 
We say that $(\tvr, \tvu,\tvt)$ is the extension of the strong solution $(\vr, \vu,\vt)$ if
\begin{equation}\label{EXT1}
(\tvr, \tvu,\tvt)(t,x) =
\begin{cases}
(\vr_{0}^s , \, 0, \, \vt_B) & \mbox{if} \ \vx \in \Os,\\
(\vr, \, \vu,\vt) &\mbox{if} \ \vx \in \Of,
\end{cases} \quad \mbox{ for any } t \in [0,T].
\end{equation}
Here $\vr_0^s=\vr_0(x)$ is the initial extension given in \eqref{ppde_ic} such that $\vr_0^s \in W^{2,\infty}(\Os), \, \tvr_{0} \in W^{1,\infty}(\tor)$.
\end{Definition}

\begin{Remark}
 With the above definition and Remark \ref{rmk-strong}
we obtain the following regularity of $\tvr, \tvu, \tvt$: $(\tvr, \tvu, \tvt)|_{\Os} \in W^{2,\infty}(\Os;\R^{d+2})$ and $(\tvr, \tvu, \tvt) \in L^{\infty}(0,T;W^{1,\infty}(\tor;\R^{d+2}))$, $\partial_t (\tvr, \tvu, \tvt) \in L^{\infty}((0,T)\times \tor;\R^{d+2})$, $ \partial_t^2 (\tvr, \tvu, \tvt) \in L^{2}((0,T)\times \tor;\R^{d+2}).$
\end{Remark}

\section{Error estimates}\label{sec_EE}
In this section we analyze the error between a numerical solution $(\vrh,\vuh,\vth)$ of the penalty FV method \eqref{VFV} and a target strong solution $(\tvr,\tvu,\tvt)$ by means of the relative energy functional
\begin{align*}
\RE{\vrh,\vuh,\vth}{\tvr,\tvu,\tvt}
= \intTdB{\frac12 \vrh |\vuh - \tvu|^2 + \REH{\vrh,\vth}{\tvr,\tvt}}
\end{align*}
with
\begin{align*}
\REH{\vr,\vt}{\tvr,\tvt} = \Hc(\vr,\vt) - \frac{\pd \Hc(\tvr,\tvt)}{\pd \vr}(\vr -\tvr) - \Hc(\tvr,\tvt) \quad \mbox{and} \quad
\hHc(\vr,\vt) = \vr( c_v \vt - \hvt s(\vr,\vt) ).
\end{align*}

\begin{Remark}
Under the assumption on the boundedness of the discrete density and temperature \eqref{HP} it holds that
\begin{equation}\label{EN}
\begin{aligned}
& \RE{\vrh,\vuh,\vth}{\tvr,\tvu,\tvt} \approx \norm{ (\vrh,\vu_h,\vth) - (\tvr,\tvu,\tvt) }_{L^2(\tor)}^2.
\end{aligned}
\end{equation}
\end{Remark}

\begin{Theorem}[{\bf Error estimates I}]\label{thm_EEB}
 Let $\Of \subset \R^d, d=2,3,$ be a bounded domain with a smooth boundary $\partial \Of \in C^k$, $k \geq 4$. Let the initial/boundary data $(\vr_0, \vu_0,\vt_0,\vtB) $ satisfy \eqref{ic-compat}.
Suppose that the NSF system \eqref{pde} admits a strong solution $(\vr,\vu,\vt)$ in the sense of Definition \ref{def-ss}.
Let $(\vrh, \vuh,\vth)$ be a numerical solution obtained by the penalty FV method \eqref{VFV} with $(\TS,h,\penl) \in (0,1)^3$ and $0 \leq \alpha <1$.
In addition, let the numerical density $\vrh$ and temperature $\vth$ be uniformly bounded, i.e. there exist $\Un{\vr}, \Ov{\vr}, \Uvt, \Ovt $ such that
\begin{equation*}
		 0< \Un{\vr} \leq \vrh \leq \Ov{\vr} ,\ 0< \Uvt \leq \vth \leq \Ovt \ \mbox{ uniformly for }\ \TS,h,\penl \to 0.
\end{equation*}

Then there exists a positive number
\begin{align*}
 c=c \big( T,\Un{\vr}, \Ov{\vr}, \Un{\vt}, \Ov{\vt},& \norm{( \vr, \vu, \vt)}_{L^{\infty}(0,T;W^{2,\infty}(\Of;\R^{d+2}))}, \norm{\partial_t( \vr, \vu, \vt)}_{L^{\infty}((0,T)\times \Of;\R^{d+2})},
\br
& \norm{\partial_t^2( \vr, \vu, \vt)}_{L^{2}((0,T)\times \Of;\R^{d+2})}, \norm{(\vr_0^s,\vt_B)}_{W^{2,\infty}(\Os;\R^2)}\big)
\end{align*}
such that it holds for any $\tau \in [0,T]$
\begin{align*}
&\RE{\vrh,\vuh,\vth}{\tvr,\tvu,\tvt}(\tau) +\frac{1}{\penl} \intTauOshB{ |\vuh|^2 + (\vth-\vthB)^2 }
\\&\quad + \intTauTd{\left| \Dhuh - \bD(\tvu) \right|^2 } + \intTauTd{|\Gradd\vth-\Grad\tvt|^2}
\\& \leq c( \TS + h^{1+\alpha} + h^{(1-\alpha)/2} + \penl^{-1} h^3 +\penl h^{-1}) \leq c( \TS + h^{(1-\alpha)/2} + \penl^{-1} h^3 +\penl h^{-1}).
\end{align*}
Consequently, we obtain the optimal error $(\TS + h^{1/2})$ for the relative energy with $\alpha =0$ and $\penl = h^2$.
Furthermore, we have
\begin{align*}
&\norm{ (\vrh,\vu_h,\vth) - (\tvr,\tvu,\tvt) }_{L^\infty(0,T; L^2(\tor;\R^{d+2}))} + \norm{(\Dhuh,\Gradd \vth) - (\bD(\tvu),\Grad \tvt) }_{L^2((0,T)\times \tor;\R^{(d+1)\times d})}
\br
& \aleq \TS^{1/2} + h^{(1-\alpha)/4} + \penl^{-1/2} h^{3/2} +\penl^{1/2} h^{-1/2}.
\end{align*}
\end{Theorem}

\begin{proof}
We present the main idea of the proof and postpone the technical details of calculations to Appendix \ref{app-ee}.

Firstly, we derive the relative energy inequality, cf. \eqref{REI-1}, by summing the consistency formulations \eqref{cP1-tor}, \eqref{cP2-tor} and \eqref{cP4}, where the test functions are taken as $\left(\frac{|\tvu|^2}{2}-\frac{\pd \Hc(\tvr,\tvt)}{\pd \vr}, \tvu, \tvt \right)$, respectively.

Secondly, we estimate the right-hand-side (RHS) of relative energy inequality \eqref{REI-1}, which includes the consistency errors and the compatibility errors.
We point out that the consistency errors estimates stated in \eqref{ec1-tor}, \eqref{ec2-tor}, and \eqref{ec4} require the test functions to be in the class of at least $L^\infty{(0,T;W^{2,\infty}(\tor))}$, which is not satisfied by the strong solution $(\tvr,\tvu,\tvt)$. We solve this problem by a suitable mesh splitting as stated in Definition \ref{def_ES}. Consequently, our aim is to control all RHS terms either directly or by the left-hand-side terms, see Appendices \ref{app-cer} and \ref{app-cer-1}.

Finally, we reach the desired version of the relative energy inequality, cf. \eqref{RE04}. Applying Gronwall's lemma we finish the proof.
\end{proof}

\begin{Remark}
 As a byproduct of Theorem \ref{thm_EEB} we obtain that our numerical approximations  converge strongly to the strong solution on its lifespan. This result holds  under the assumption on the boundedness of numerical density and temperature, cf. \eqref{HP}.
\end{Remark}

\begin{Theorem}[{\bf Error estimates II}]\label{thm_EEB_1}
Let $\Of \subset \R^d, d=2,3,$ be a bounded domain with a smooth boundary $\partial \Of \in C^k$, $k \geq 4$. Let the initial/boundary data $(\vr_0, \vu_0,\vt_0,\vtB) $ satisfy \eqref{ic-compat}.

Let $(\vrh, \vuh, \vth)$ be a numerical solution obtained by the penalty FV method \eqref{VFV} with $(\TS,h,\penl) \in (0,1)^3$ and $-1 < \alpha <1$.
Under the assumption
\begin{align*}
 0< \Un{\vr} \leq \vrh \leq \Ov{\vr}, \quad
 0< \Un{\vt} \leq \vth \leq \Ov{\vt}, \quad
  \abs{\vuh} \leq \Ov{u} \quad \mbox{ uniformly for }\ \TS,h,\penl \rightarrow 0
\end{align*}
we have
\begin{align*}
&\RE{\vrh,\vuh,\vth}{\tvr,\tvu,\tvt}(\tau) +\frac{1}{\penl} \intTauOshB{ |\vuh|^2 + (\vth-\vthB)^2 }
\\&\quad + \intTauTd{\left| \Dhuh - \bD(\tvu) \right|^2 } + \intTauTd{|\Gradd\vth-\Grad\tvt|^2}
\\& \leq c( \TS + h^{1+\alpha} + h^{1-\alpha} + \penl^{-1} h^3 +\penl h^{-1})
\end{align*}
and
\begin{align*}
&\norm{ (\vrh,\vu_h,\vth) - (\tvr,\tvu,\tvt) }_{L^\infty(0,T; L^2(\tor;\R^{d+2}))} + \norm{(\Dhuh,\Gradd \vth) - (\bD(\tvu),\Grad \tvt) }_{L^2((0,T)\times \tor;\R^{(d+1)\times d})}
\br
& \leq c( \TS^{1/2} + h^{(1+\alpha)/2} + h^{(1-\alpha)/2} + \penl^{-1/2} h^{3/2} +\penl^{1/2} h^{-1/2}),
\end{align*}
where
$$ c=c \left( T,\Un{\vr}, \Ov{\vr}, \Un{\vt}, \Ov{\vt},\norm{(\vr_0,\vu_0,\vt_0)}_{W^{4,2}(\Of;\R^{d+2})}, \norm{\vt_B}_{W^{9/2,2}(\partial\Of)}, \norm{(\vr_0^s,\vt_B)}_{W^{2,\infty}(\Os;\R^2)}\right)
$$
and
$(\vr, \vu,\vt)$ is the strong solution of the Dirichlet problem \eqref{pde} in the sense of Definition \ref{def-ss}. Moreover, we obtain the optimal error $(\TS + h)$ for the relative energy with $\alpha =0$ and $\penl = h^2$,  which is improved in comparison to Theorem \ref{thm_EEB}.
\end{Theorem}
\begin{proof}
The optimal first order convergence rate is improved in two aspects.
Firstly, with $\abs{\vuh} \aleq 1$ we improve the terms of order $h^{(1-\alpha)/2}$ to $h^{1-\alpha}$, which origins from $E_2(\vrh \vuh, \tvu) $ and $I_{\vu,2}$, cf. \eqref{E_nf} and \eqref{Iu2}, respectively. The term $E_2(\vrh \vuh, \tvu) $ can be now estimated as follows
\begin{align*}
&\abs{E_2(\vrh \vuh, \tvu)} = \frac14 \abs{
\intTau{\intfaces{ \jump{\vuh} \cdot \vc{n} \jump{\vrh \vuh } \jump{ \Pim \tvu} } }
}
\\&
\aleq
\abs{
\intTau{\intfaces{ \jump{\vuh} \cdot \vc{n} \jump{\vrh }\avs{\vuh} \jump{ \Pim \tvu} } }
}+
\abs{
\intTau{\intfaces{ \jump{\vuh} \cdot \vc{n} \avs{\vrh}\jump{ \vuh } \jump{ \Pim \tvu } } }
}
\\& \aleq
h\left(\intTau{\intfaces{ \jump{\vuh}^2} }\right)^{1/2}\left(\intTau{\intfaces{ \jump{\vrh}^2} }\right)^{1/2}
+h \intTau{\intfaces{ \jump{\vuh}^2} }
 \aleq h^{1-\alpha} .
\end{align*}
Moreover, thanks to the regularity of the strong solution $\vu \in L^2(0,T;W^{5,2}(\Of))$, applying estimates \eqref{EXTE} -- \eqref{lm_ms} we improve the estimate of the compatibility error $I_{\vu,2} $ as follows
\begin{align*}
&\abs{I_{\vu,2}} = \abs{ \intTauTd{\vuh \cdot ( \Divmesh \Piw \tbS - \Divh \Pim \tbS ) }}
\br&
\aleq h \norm{\vu}_{L^1(0,T;W^{3,2}(\Of;\R^d))} \norm{\vuh}_{L^\infty(0,T;L^2(\tor;\R^d))} + h^{-1} \norm{\tvu}_{L^\infty(0,T;W^{1,\infty}(\tor;\R^d))} \intTauOch{\abs{\vuh}}
\br&
\aleq h + h^{-1} \left( \intTau{ \int_{\Omega_h^{C,1}} h\abs{\Gradh \vuh} \dx } + \intTau{ \int_{\Omega_h^{C,2}} \abs{\vuh} \dx} \right)
\br&
\aleq h^{(1-\alpha)}+ h+ \delta \norm{ \Dhuh - \bD(\tvu)}_{L^2((0,\tau) \times\tor)}^2 + \frac{\penl}{h} + \delta \frac{\norm{\vuh}_{L^2((0,\tau)\times\Osh)}^2}{\penl}.
\end{align*}
Here $\Omega_h^{C,1}$ (resp. $\Omega_h^{C,2}$) is obtained by shifting $\Och$ one cell (resp. two cells) towards $\Os$. We note that we have used $\Omega_h^{C,2} \subset \Omega_h^{C,1} \subset \Osh$ for the last inequality.

Secondly, we improve the term of order $h^{\alpha+1}+ h^\alpha \intTau{\RE{\vrh,\vuh,\vth}{\tvr,\tvu,\tvt}}$ to $h^{1+\alpha}$, which origins from
 $E_3(\vrh \vuh, \phi) $.
Applying $\abs{\vuh} \aleq 1$ together with \eqref{EXTD} and \eqref{EXTE4} we get a new estimate:
\begin{align*}
&\abs{E_3(\vrh \vuh, \tvu )}
=h^{\alpha+1}\abs{ \intTauTd{ \vrh \vuh \cdot \Delta_h \Pim \tvu }}
 \aleq h^{\alpha+1} \intTauTd{ \abs{ \Delta_h \Pim \tvu } } \aleq h^{1+\alpha}.
\end{align*}
Consequently, we do not need the constraint of $\alpha \geq 0$ and complete the proof.
\end{proof}
\medskip

\begin{Remark}
We point out that our previous result \cite[Theorem 5.2]{Basaric} can be improved by controlling $R_s$ with $D_s$, cf. \eqref{ee-rs}.
Specifically,
letting $\Of \equiv \tor$ we obtain for the NSF system with periodic boundary conditions the following results:
\begin{itemize}
\item Under the same condition as Theorem \ref{thm_EEB} it holds
\begin{equation*}
\begin{aligned}
&
\RE{\vrh,\vuh,\vth}{\tvr,\tvu,\tvt}(\tau) + \intTauTd{\left| \Dhuh - \bD(\tvu) \right|^2 } + \intTauTd{|\Gradd\vth-\Grad\tvt|^2}
\\& \leq c( \TS + h^{1+\alpha} + h^{(1-\alpha)/2}) \quad \quad \mbox{for }\; \alpha \in (-1,1).
\end{aligned}
\end{equation*}
Consequently, taking $\alpha = -1/3$ we obtain the optimal error $(\TS + h^{2/3})$ for the relative energy.

 \item Under the assumption of Theorem \ref{thm_EEB_1} it holds
\begin{equation*}
\begin{aligned}\label{ee_re-1}
&
\RE{\vrh,\vuh,\vth}{\tvr,\tvu,\tvt}(\tau) + \intTauTd{\left| \Dhuh - \bD(\tvu) \right|^2 } + \intTauTd{|\Gradd\vth-\Grad\tvt|^2}
\\& \leq c( \TS + h^{1+\alpha} + h^{1-\alpha}) \quad  \quad \mbox{for }\; \alpha \in (-1,1).
\end{aligned}
\end{equation*}
Consequently, taking $\alpha = 0$ we obtain the optimal error $(\TS + h)$ for the relative energy.

 \end{itemize}

\end{Remark}

\section*{Acknowledgements}
The authors sincerely thank E.~Feireisl (Prague) for stimulating discussions.

\appendix
\section{Proof of the positivity of density and temperature}\label{app-positivity}
In this section we prove the positivity of density and temperature.
To this end, we present firstly the corresponding renormalized equations.  Specifically, they are the renormalized continuity equation \cite[Lemma 8.3]{FeLMMiSh}
and the renormalized internal energy equation \cite[Lemma 3.3]{HS_NSF}.
To begin, we introduce the notation
\begin{equation*}
 E_f(v_1|v_2) = f(v_1) - f'(v_2)(v_1-v_2) - f(v_2), \quad f \in C^1(\R).
\end{equation*}
\begin{Lemma}[Renormalized continuity equation {\cite[Lemma 8.3]{FeLMMiSh}}]\label{lem_r2}
Let $(\vrh ,\vuh)$ satisfy \eqref{VFV_D}. Then for any $\phi_h \in Q_h$ and any function $B\in C^1(\R)$
we have
\begin{align} \label{renormalized_density}
& \intTd{ {\rm D}_t B(\vrh) \phi_h } - \intfaces{ \Up[B(\vrh), \vuh] \cdot \jump{\phi_h} } + \intTd{ \phi_h \; (\vrh B'(\vrh) - B(\vrh) ) \; \Divh \vuh }
\nonumber \\
& = - \frac{1}{\Delta t}\intTd{\phi_h \EB{\vrh^\triangleleft |\vrh} } - h^\alpha \intfaces{ \jump{\vrh} \jump{B'(\vrh) \phi_h}}
-\intfaces{ |\avs{\vuh } \cdot \vn | \phi_h^{\rm down} \EB{\vrhup|\vrhdown} }.
\end{align}
\end{Lemma}

\begin{Lemma}
[Renormalized internal energy equation {\cite[Lemma 3.3]{HS_NSF}}]\label{lem_r3}
Let $(\vrh , \vuh , \vth )$ satisfy equation \eqref{VFV}. Then for any $\phi_h \in Q_h$, and any function $\chi \in C^1(\R)$ it holds
\begin{align}\label{eq_renormalized_energy}
& c_v\intTd{ {\rm D}_t (\vrh \chi(\vth) ) \phi_h } - c_v\intfaces{ \Up[\vrh \chi(\vth), \vuh] \cdot \jump{ \phi_h} } + \intfaces{ \frac{\kappa}{h} \jump{\vth} \jump{\chi'(\vth) \phi_h} }
\nonumber \\
= &\intTd{ (\bS_h - p_h\bI): \Gradh \vuh \chi'(\vth) \phi_h} - \frac{c_v}{\Delta t} \intTd{\vrh^{\vartriangleleft} \phi_h\Echi{\vth^\triangleleft|\vth} }
\nonumber \\
& - c_v h^\ep \intfacesB{ \jump{\vrh} \jump{\big(\chi(\vth) - \chi'(\vth) \vth\big) \phi_h } + \jump{\vrh \vth} \jump{\chi'(\vth) \phi_h } }
\nonumber \\
& - c_v\intfaces{ |\avs{\vuh } \cdot \vn | \phi_h^{\rm down} \vrhup \Echi{\vthup|\vthdown} }
- \frac{1}{\penl}\intOsh{ (\vth-\vthB) \chi'(\vth ) \phi_h }.
\end{align}
\end{Lemma}
\begin{Remark}\label{rmq1}
Considering $\vr$ and $p$ as the independent variables, e.g., $\vt=p/\vr$, we know that
\begin{equation*}
\nabla_{\vr}\left( \vr \chi( \vt ) \right) = \chi(\vt) - \chi'(\vt) \vt, \quad
\nabla_p\left( \vr \chi( \vt ) \right) = \chi'(\vt),
\end{equation*}
which leads to the following identity
\begin{align*}
&\intfacesB{ \jump{\vrh} \jump{\big(\chi(\vth) - \chi'(\vth) \vth\big) \phi_h } + \jump{\vrh \vth} \jump{\chi'(\vth) \phi_h } }
\\&= \intfacesB{ \jump{\vrh} \jump{\nabla_\vr(\vrh \chi(\vth)) \phi_h } + \jump{p_h} \jump{\nabla_p(\vrh \chi(\vth)) \phi_h } }.
\end{align*}
Furthermore, setting $\chi(\vt) = ([\vt]^-)^2$ it holds $\jump{\vrh} \jump{\nabla_\vr(\vrh \chi(\vth)) \phi_h } + \jump{p_h} \jump{\nabla_p(\vrh \chi(\vth))} \geq 0 $.

\end{Remark}

\subsection{Proof of Lemma \ref{lem_p1}}
We proceed to analyze the positivity of density and temperature.
For this purpose, viewing $(\vr, \vt)$ as the independent variables we continuously extend the pressure to the non-physical temperature regime ($\vt \leq 0$) as in our previous paper \cite{FLMS_FVNSF}
\begin{equation} \label{gas_law}
p= \vr \vt \quad \mbox{ for } \vt >0 \quad \mbox{ and } \quad p = 0 \quad \mbox{ for } \vt \leq 0.
\end{equation}
We point out that Lemma \ref{lem_p1} (or Lemma \ref{lem_Ap1}) confirms that our numerical approximations $\vrh$ and $\vth$ are always located in the physical regime, i.e. $\vrh> 0$ and $\vth>0$.

\begin{Lemma}[Positivity of density and temperature] \label{lem_Ap1}
Let $(\vrh ,\vuh ,\vth )$ be a numerical solution of the FV method \eqref{VFV}.
Then it holds
\begin{align*}
\vrh(t) > 0, \quad \vth(t) > 0 \quad \mbox{for all } \ t\in(0,T).
\end{align*}
Note that $\vrh^0 > 0$ and $\vth^0 > 0$ follow from the positivity of the data $\tvr^0 > 0$ and  $\tvt^0 > 0$.
\end{Lemma}
\begin{proof}
The proof of $\vrh>0$ has already been shown in \cite[Lemma 11.2]{FeLMMiSh} by means of renormalized continuity equation \eqref{renormalized_density}. In what follows we concentrate on the positivity of $\vth$.

Setting $\phi_h = 1$ and $\chi(\vt) = \frac 1 2 ([\vt]^-)^2$ in \eqref{eq_renormalized_energy} and recalling Remark \ref{rmq1} we obtain
\begin{align}\label{eq_1}
& c_v\intTd{ {\rm D}_t (\vrh \chi(\vth) ) } = - \intfaces{ \frac{\kappa}{h} \jump{\vth} \jump{\chi'(\vth) } } + \intTd{ (\bS_h - p_h\bI): \Gradh \vuh \chi'(\vth)}
\nonumber \\
&
\quad-\frac{c_v}{\Delta t} \intTd{\vrh^{\vartriangleleft} \Echi{\vth^\triangleleft|\vth} \phi_h }
- c_v\intfaces{ |\avs{\vuh } \cdot \vn | \phi_h^{\rm down} \vrhup \Echi{\vthup|\vthdown} }
\nonumber
\\&
\quad- \frac{1}{\penl}\intOsh{ (\vth-\vthB) \chi'(\vth ) }
 - c_v h^\ep \intfacesB{ \jump{\vrh} \jump{\nabla_{\vr}\left( \vrh \chi( \vth ) \right) } + \jump{p_h} \jump{\nabla_{p}\left( \vrh \chi( \vth ) \right) } }.
\end{align}
It is easy to check that
\begin{equation*}
\chi'(\vt) = \begin{cases}
0 & \mbox{if} ~\vt \geq 0,\\
\vt & \mbox{if} ~ \vt < 0.
\end{cases}
\end{equation*}
Hence, $\chi(\vt)$ is a $C^1(\R)$-convex function satisfying
\begin{align*}
\jump{\vth} \jump{\chi'(\vth) } \geq 0,\quad
\chi'(\vt) \leq 0,
\quad \chi'(\vt) p(\vr,\vt) \equiv 0, \quad \chi'(\vt) (\vt - \vtB) \geq 0.
\end{align*}
Recalling that $\bS_h : \Gradh \vuh = 2\mu \abs{\bD_h \vuh}^2 + \lambda \abs{\Divh \vuh}^2$, $E_\chi$ is convex and Remark \ref{rmq1},
we know that all terms on the right-hand-side of \eqref{eq_1} are nonpositive.
Consequently, we obtain
\begin{equation*}
\intTd{ D_t \left(\vrh \chi(\vth ) \right) } \leq 0,
\end{equation*}
which implies $\vth \geq 0$ provided $\vth^\triangleleft \geq 0.$

To conclude the proof it remains to show that $\vth \neq 0$, which we prove by contradiction. Given $\vth^\triangleleft > 0$ we assume that there exists an element $K$ such that
$\vth \big |_K = 0$. Then with the definition of pressure \eqref{gas_law} we have $p_h\big |_K = 0$ and
 the internal energy equation \eqref{VFV_E} directly gives
\begin{align*}
&-\frac{c_v}{\TS} (\vrh^\triangleleft \vth^\triangleleft)_K + c_v \sum_{ \sigma \in \pd K } \frac{|\sigma|}{|K|} (\vrh \vth )^{\rm out}
\left( [\vuh \cdot\vn]^- - h^\alpha \right)
 - \frac{\kappa}{h^2}\sum_{ \sigma \in \pd K } \frac{|\sigma|}{|K|} \vth ^{\rm out}
\br
&= 2 \mu |\Dhuh|_K^2 + \lambda |\Divh \vuh |_K^2
+ \frac{\mathds{1}_{\Osh}}{\penl} \vthB|_K,
\end{align*}
leading to the negativity of the left-hand-side and the non-negativity of the right-hand-side.
That is a contradiction and concludes the proof.
\end{proof}

\section{Some useful estimates}\label{app-ue}
In this section we prove some estimates which will be useful for the convergence analysis and error estimates, i.e. Appendices \ref{app-cf} and \ref{app-ee}.
We start by recalling the projection error of the discrete differential operators
\begin{equation}\label{proj}
\abs{ \jump{\Pim \phi } }\aleq h \norm{\Grad \phi}_{L^\infty(\tor)}, \;
\norm{ \Pim \phi - \phi }_{L^\infty(\tor)} \aleq h \norm{\Grad \phi}_{L^\infty(\tor)}, \;
\norm{\Gradd \Pim \phi }_{L^\infty(\tor)} \aleq \norm{ \Grad \phi }_{L^\infty(\tor)}
\end{equation}
for any $\phi \in W^{1,\infty}(\tor)$, cf. \cite{FeLMMiSh}. Further, the following discrete integration by parts formulae hold
\begin{subequations}\label{dis_op}
\begin{align}
& \intTd{ r_h\Divh \vv_h } = -\intTd{ \Gradh r_h \cdot \vv_h },
\\
& \intTd{ \Delta_h r_h \cdot f_h } = -\intTd{ \Gradd r_h \cdot \Gradd f_h } = \intTd{r_h \cdot \Delta_h f_h }, \label{dis_op3}
\\
&\intTd{ r_h\Div \bfphi } =    \intTd{ r_h\Divmesh \Piw \bfphi } =    - \intTd{ \Gradd r_h \cdot \Piw \bfphi }
\end{align}
\end{subequations}
for any $r_h, f_h \in Q_h$ and $\bfphi \in W^{1,1}(\tor;\R^d)$, see the preliminary chapter of Feireisl et al. \cite{FeLMMiSh}.
Note that $\Piw$ is a projection operator given by
\begin{align*}
\Piw = (\Piw^{(1)}, \cdots, \Piw^{(d)}),
\quad
\Piwi \phi(x) = \sum_{\sigma \in \facesi} \frac{\mathds{1}_{D_\sigma}(x)}{|\sigma|} \int_{\sigma} \phi \ds \quad \quad
 \mbox{for any } \phi\in W^{1,1}(\tor)
\end{align*}
and 
$\Divmesh$ is a discrete divergence operator  given by 
\[
\Divmesh  \bfphi_h  = \sumi \pdmeshi \Phi_{i,h}, \quad  \pdmeshi \Phi_{i,h}|_K = \frac{  \Phi_{i,h}|_{\sigma_K^{i+}} -\Phi_{i,h}|_{\sigma_K^{i-}}   }{h} , 
\]
where $\bfphi_h = ( \Phi_{i,h}, \ldots,  \Phi_{d,h})^T$, $\sigma_K^{i-}$ and $\sigma_K^{i+}$ are respectively the left and right faces of an element $K$ that are orthogonal to the canonical basis vector $\ve_i$.

Next, we present the properties of the extended strong solution $(\tvr,\tvu,\tvt)$ defined in \eqref{EXT1}.
\begin{Lemma}
Let $(\tvr,\tvu,\tvt)$ be the extended strong solution defined in \eqref{EXT1}. Then there hold
\begin{subequations}\label{EXTE}
\begin{equation}\label{EXTE1}
 \abs{\tvu} \aleq \norm{\tvu}_{W^{1,\infty}(\tor)} h \mbox{ on } \Och; \quad
 \abs{ \Gradh( \Pim \tvu)} + \abs{ \Gradd( \Pim \tvu)} + \abs{ \Laph \Pim \tvu } = 0 \mbox{ on } \ \OOh,
\end{equation}
\begin{equation}\label{EXTE2}
\abs{ \Gradh( \Pim f)} + \abs{\Divh(\Pim f)} + \abs{ \Gradd( \Pim f)} + \abs{ \Piw( \Grad f)} \aleq \norm{f}_{W^{1,\infty}(\tor)} \mbox{ on }\ \tor ,
\end{equation}
\begin{equation}\label{EXTE4}
\abs{ \Laph \Pim f } \aleq \begin{cases}
\norm{f}_{W^{2,\infty}(\Of)} + \norm{f}_{W^{2,\infty}(\Os)} & \mbox{on } \ \tor \setminus \Och, \\
\norm{f}_{W^{1,\infty}(\tor)} h^{-1} & \mbox{on } \ \Och,
\end{cases}
\end{equation}
\begin{equation}
\label{lm_f2}
\intTd{ \Abs{v \left(\Grad f - Df \right) }}
\aleq h\intTd{ \abs{v}} + \intOch{ \abs{v}}, \quad Df= \Piw \Grad f, \Gradd \Pim f ,\Gradh \Pim f,
\end{equation}
\begin{equation}\label{EXTE5}
\norm{\Gradh \vuh - \Grad \tvu}_{L^2(\tor)} \aleq h + \norm{\Dhuh - \bD(\tvu)}_{L^2(\tor)},
\end{equation}
where $f = \tvr,\, \tvu$ or $\tvt$.
\end{subequations}
\end{Lemma}
\begin{proof}
Estimates \eqref{EXTE1} -- \eqref{EXTE4} directly follow from the definition of $(\tvr,\tvu,\tvt)$.
Further, combining the interpolation error and \eqref{EXTE2} we obtain
\begin{align*}
&\intTd{ \Abs{v \left(\Grad f - D f \right) }} = \int_{\OOh \cup \Och \cup \OIh}{ \Abs{v \left(\Grad f - D f \right) }} \dx
\br
&\aleq h \left(\norm{f}_{W^{2,\infty}(\Of)} + \norm{f}_{W^{2,\infty}(\Os)} \right)\int_{\OOh \cup \OIh}{ \Abs{v} } \dx + \norm{f}_{W^{1,\infty}(\tor)} \intOch{ \Abs{v} } \aleq h\intTd{ \abs{v}} + \intOch{ \abs{v}}.
\end{align*}
 Moreover, considering $v=\Grad f - D f $ in the above estimate we obtain
\[ \norm{\Gradh \Pim \tvu - \Grad \tvu}_{L^2(\tor;\R^d)} + \norm{\bD_h(\Pim \tvu) - \bD( \tvu)}_{L^2(\tor;\R^{d\times d})}\aleq h.\]
Further, recalling \eqref{S6} and thanks to the above estimate we obtain
\begin{align*}
&\norm{\Gradh \vuh - \Grad \tvu}_{L^2(\tor)} \aleq \norm{\Gradh \Pim \tvu - \Grad \tvu}_{L^2(\tor)} + \norm{\Gradh (\vuh - \Pim \tvu)}_{L^2(\tor)}
\br
&\aleq h + \norm{\bD_h(\vuh-\Pim \tvu)}_{L^2(\tor)} \aleq h + \norm{\Dhuh - \bD(\tvu)}_{L^2(\tor)} + \norm{\bD_h(\Pim \tvu) - \bD( \tvu)}_{L^2(\tor)}
\br
&\aleq h + \norm{\Dhuh - \bD(\tvu)}_{L^2(\tor)}
\end{align*}
and finish the proof.
\end{proof}

Next, we estimate the numerical approximations $(\vrh, \vuh, \vth)$ on the outer domain $\Osh$, on the strip area $\Och$, and finally on the whole domain $\tor$.
\begin{Lemma}
Let $(\vrh, \vuh, \vth)$ be a numerical solution obtained by the FV scheme \eqref{VFV} with $(\TS,h,\penl) \in (0,1)^3$ and $ \alpha >-1$. Let $\delta>0$ be an arbitrary constant.
Then for any $\tau \in [0,T]$ the following hold
\begin{subequations}\label{lm}
\begin{align} \label{lm_u0}
\int_0^\tau \int_{\Osh \setminus \Os} {\abs{\vuh}} \dxdt \leq
\int_0^\tau \int_{\Osh \setminus \OOh} {\abs{\vuh}} \dxdt
\aleq \ & \delta h \frac{\norm{\vuh}_{L^2((0,\tau)\times\Osh)}^2}{\penl} + \frac{\penl}{\delta},
\\ \label{lm_u1}
\intTauOs{\abs{\vuh}} \aleq \
\intTauOsh{\abs{\vuh}} \aleq \ & \delta \frac{\norm{\vuh}_{L^2((0,\tau)\times\Osh)}^2}{\penl} + \frac{\penl}{\delta},
\\ \label{lm_vt0}
\int_0^\tau \int_{\Osh \setminus \OOh} {\abs{\vth-\vthB}} \dxdt
\aleq \ & \delta h \frac{\norm{\vth-\vthB}_{L^2((0,\tau)\times\Osh)}^2}{\penl} + \frac{\penl}{\delta} ,
\\ \label{lm_vt1}
\intTauOsh{\abs{\vth-\vthB}} \aleq \ & \delta \frac{\norm{\vth-\vthB}_{L^2((0,\tau)\times\Osh)}^2}{\penl} + \frac{\penl}{\delta}.
\end{align}
In addition, let assumption \eqref{HP} hold, then
\begin{align}\label{lm_0}
& \int_0^\tau \int_{\Osh \setminus \OOh} {\abs{\vuh}} \dxdt \aleq h^{1/2} \penl^{1/2},
\quad \int_0^\tau \int_{\Osh \setminus \OOh} {\abs{\vth-\vthB}} \dxdt
\aleq \ \min(h^{1/2} \penl^{1/2}, h),\\ \label{lm_1}
& \intTauOsh{\abs{\vuh}} \aleq \penl^{1/2},
\quad \intTauOsh{\abs{\vth-\vthB}} \aleq \penl^{1/2}.
\end{align}
\end{subequations}
\end{Lemma}
\begin{proof}
Firstly, we get the first estimate of \eqref{lm_u0} and \eqref{lm_u1} respectively
from the fact that $\Osh \setminus \Os \subset \Osh \setminus \OOh$ and $\Os \subset \Osh$ stated in \eqref{EXTE0}.

For the second estimate of \eqref{lm_u0} and \eqref{lm_u1} we apply Young's inequality, i.e.
 \begin{align*}
 \int_0^\tau \int_{\Osh \setminus \OOh} {\abs{\vuh}} \dxdt
&\leq \int_0^\tau \int_{\Osh \setminus \OOh} {\left(\frac{\delta h}{2\penl} \abs{\vuh}^2 + \frac{\penl}{2\delta h}\right) } \dxdt
\\&
\leq \frac{\delta h}{2\penl} \intTauOsh {\abs{\vuh}^2 }
+
\intTauOch { \frac{\penl}{2\delta h} }
 \aleq \delta h \frac{\norm{\vuh}_{L^2((0,\tau)\times\Osh)}^2}{\penl} + \frac{\penl}{\delta},
\end{align*}
\begin{align*}
\intTauOsh{ \abs{ \vuh} }
\leq \intTauOsh{\left(\frac{\delta }{2\penl} \abs{\vuh}^2 + \frac{\penl}{2\delta }\right) }
\aleq \delta \frac{\norm{\vuh}_{L^2((0,\tau)\times\Osh)}^2}{\penl} + \frac{\penl}{\delta} .
\end{align*}
The same process applies to $\vth-\vthB$ that yields \eqref{lm_vt0} and \eqref{lm_vt1}.

Further, letting $\delta = (\penl/h)^{1/2}$ in \eqref{lm_u0} and \eqref{lm_vt0}, we directly obtain \eqref{lm_0} from assumption \eqref{HP} and the uniform bounds \eqref{ap2}.
Setting $\delta = \penl^{1/2}$ in \eqref{lm_u1} and \eqref{lm_vt1} leads to \eqref{lm_1}, which completes the proof.
\end{proof}

\begin{Lemma}
Let $(\vrh, \vuh, \vth)$ be a numerical solution obtained by the FV scheme \eqref{VFV} with $(\TS,h,\penl) \in (0,1)^3$ and $ \alpha >-1$. Let $\delta>0$ be an arbitrary constant.
Then the following holds
\begin{subequations}\label{lm_ms}
\begin{align}\label{lm_u1_c}
\intOch{ \abs{ \vuh } } & \aleq h+ \RE{\vrh,\vuh,\vth}{\tvr,\tvu,\tvt},
\\ \label{lm_u2_c}
 \intOch{ \abs{\vuh}^2 } & \aleq h^3+ \RE{\vrh,\vuh,\vth}{\tvr,\tvu,\tvt},
\\ \label{lm_gradu_c}
 \intOch{\abs{\Dhuh}} & \aleq \frac{h}{\delta} + \delta \intOch{|\Dhuh - \bD(\tvu)|^2},
\\ \label{lm_gradvt_c}
\intOch{\abs{ \Gradd \vth } }
& \aleq  \frac{h}{\delta} + \delta \intOch{ |\Gradh \vth - \Grad \tvt|^2 }.
\end{align}
Further, we have
\begin{equation}\label{lm_gradu_c-1}
\intOch{\abs{ \Gradh \vuh} } \aleq  \frac{h}{\delta} + \delta \intOch{ |\Gradh \vuh - \Grad \tvu|^2 } \aleq  \frac{h}{\delta} + \delta \norm{\Dhuh - \bD(\tvu)}_{L^2(\tor)}^2.
\end{equation}
\end{subequations}
\end{Lemma}

\begin{proof}
By triangular inequality, Young's inequality, \eqref{EXTD}, \eqref{EN}, and \eqref{EXTE1} we obtain \eqref{lm_u1_c} and \eqref{lm_u2_c}:
\begin{align*}
\intOch{ \abs{ \vuh } }
&\leq \intOch{ \abs{ \vuh - \tvu} }
+ \intOch{ \abs{ \tvu} } \aleq \intOchB{ \abs{ \vuh - \tvu}^2 +1} + \intOch{ \abs{ \tvu} }
\\&
\aleq \intOch{ \abs{ \vuh - \tvu}^2} + (1+h) \abs{ \Och}
\aleq h+\RE{\vrh,\vuh,\vth}{\tvr,\tvu,\tvt},
\br
 \intOch{ \abs{ \vuh }^2 }
& \aleq \intOch{ \abs{ \vuh - \tvu}^2 }
+ \intOch{ \abs{ \tvu}^2 } \aleq h^{ 3}+ \RE{\vrh,\vuh,\vth}{\tvr,\tvu,\tvt}
.
\end{align*}
Analogously, applying Young's inequality to $\Dhuh$, $ |\Gradd \vth|$ and $\Gradh \vuh$, we obtain \eqref{lm_gradu_c} -- \eqref{lm_gradu_c-1}, respectively. For the last inequality of \eqref{lm_gradu_c-1} we have used \eqref{EXTE5}.
\end{proof}

\section{Proof of the consistency formulation}\label{app-cf}
In this section we prove the consistency formulation and estimate the consistency errors, cf. Lemmas \ref{lem_C1} and \ref{lem_C4}.
To begin, we recall the consistency formulation of our previous paper \cite[Lemma 4.5]{BLMSY}.
\begin{Lemma}[{\cite[Lemma 4.5]{BLMSY}}]\label{lem_C1-tor}
Let $(\vrh, \vuh, \vth)$ be a numerical solution obtained by the FV scheme \eqref{VFV} with $(\TS,h,\penl) \in (0,1)^3$, $-1 < \alpha <1$.
Let assumption \eqref{HP} hold.

Then for any $\tau \in [t^n,t^{n+1})$, $n=0,1,\cdots,N_T$  it holds
\begin{subequations}\label{eq_C1-tor}
\begin{equation} \label{cP1-tor}
\left[ \intTd{ \vrh\phi } \right]_{t=0}^{t=\tau}=
 \intTauTdB{ \vrh \partial_t \phi + \vrh \vuh \cdot \Grad \phi } + e_\vr(\phi)
\end{equation}
for any $\phi \in L^\infty(0,T;W^{2,\infty}(\tor))$, $\partial_t\phi \in L^2((0,T) \times \tor)$, $\partial^2_t\phi \in L^2((0,T) \times \tor)$;
\begin{align} \label{cP2-tor}
\left[ \intTd{ \vrh \vuh \cdot \bfphi } \right]_{t=0}^{t=\tau} = &
\intTauTdB{ \vrh \vuh \cdot \partial_t \bfphi + \vrh \vuh \otimes \vuh : \Grad \bfphi } - \frac{1}{\penl} \int_0^{t^{n+1}}\intOsh{ \vuh \cdot \bfphi }
\br
&- \intTauTd{ ( \bS_h -p_h \I) : \Grad \bfphi }
+ \widetilde{e_{\vm}}(\bfphi)
 \end{align}
 for any $\bfphi \in L^\infty(0,T;W^{2,\infty}(\tor;\mathbb R^d))$, $\partial_t\bfphi \in L^2((0,T) \times \tor; \mathbb R^d)$, $\partial^2_t \bfphi \in L^2((0,T) \times \tor; \mathbb R^d)$;
 \begin{align}\label{cP3-tor}
\left[ \intTd{ \vrh s_h \phi } \right]_{t=0}^{t=\tau} & \geq
 \int_0^\tau\intTd{ \vrh s_h (\pd_t\phi + \vuh \cdot \Grad \phi ) } \dt - \int_0^\tau\intTd{ \frac{\kappa}{\vth} \Gradd \vth \cdot \Grad \phi } \dt
\br
&+ \int_0^\tau\intTd{ \frac{ \kappa \phi}{ \vthout \vth} \abs{\Gradd \vth }^2 } \dt
+ \int_0^\tau\intTd{\difuh \frac{ \phi}{\vth} } \dt
\br
& - \frac{1}{\penl} \intn \intOsh{(\vth -\vthB)\frac{\phi}{\vth}} \dt + \widetilde{e_{s}}(\phi)
\end{align}
for any $\phi \in L^\infty(0,T;W^{2,\infty}(\tor))$, $\partial_t\phi \in L^2((0,T) \times \tor)$, $\partial^2_t\phi \in L^2((0,T) \times \tor)$, $\phi \geq 0$.

The consistency errors satisfying
\begin{align}\label{ec1-tor}
 \Bigabs{e_\vr(\phi) } & \leq C_\varrho\left( \TS +h+h^{1-\alpha} + h^{\alpha+1}\right),
\\
 \Bigabs{ \widetilde{e_{\vm}}(\bfphi) }& \leq C_{\vm}\left(\TS+h+h^{(1-\alpha)/2}+h^{\alpha+1} \right), \label{ec2-tor}
 \\
\Bigabs{ \widetilde{e_{s}}(\phi) } &\leq C_s\left( \TS +h+h^{1-\alpha} + h^{(1+\alpha)/2} \right). \label{ec3-tor}
\end{align}
\end{subequations}
Here the constants $C_\vr,$ $C_{\vm},$ $C_{s}$ are independent of parameters $\TS,h,\penl$.
\end{Lemma}

\begin{Remark}
Setting $\phi_h :=\Pim \phi$ and $\bfphi_h := \Pim \bfphi$, the above consistency errors can be defined in the following way:
\begin{align*}
&e_{\vr}(\phi) =E_t(\vrh,\phi) + E_F(\vrh,\phi) ,
\\& 
\widetilde{e_{\vm}}(\bfphi) =E_t(\vm_h,\bfphi) + E_F(\vm_h,\bfphi) +E_{\vm,\bS}(\bfphi) +E_{\vm,p}(\bfphi),
\\& 
\widetilde{e_{s}}(\phi) =E_t(\vrh s_h,\phi) + E_{s,F}(\vrh s_h,\phi) + E_{s,\Grad\vt}(\phi)+ E_{s,res}(\phi),
\end{align*}
where 
\begin{align*}
E_t(r_h,\phi) &= \left[ \intTd{r_h \phi}\right]_{t=0}^\tau - \intTauTd{ r_h \pdt \phi} - \intn \intTd{D_t r_h \phi_h }\dt ,
\\ 
E_F(r_h,\phi) &= - \intTauTd{r_h \vuh \cdot \Grad \phi} + \intn \intfaces{ \Fup (r_h,\vuh) \jump{ \phi_h}} \dt
\br
& = -\sum_{i=1}^4 \int_0^{t^{n+1}} E_i(r_h,\phi)\dt + \int_{\tau}^{t^{n+1}} \intTd{r_h \vuh \cdot \Grad \phi} \dt,
\\ 
E_{\vm,\bS}(\bfphi) & = \intTauTd{ \bS_h : \Grad \bfphi} - \intn \intTd{ \bS_h : \Gradh \bfphi_h } \dt ,
\\ 
E_{\vm,p}(\bfphi) & = - \intTauTd{ p_h \Div \bfphi } + \intn \intTd{ p_h \Divh \bfphi_h} \dt ,
\\ 
E_{s,F}(r_h,\phi) &= - \intTauTd{r_h \vuh \cdot \Grad \phi} + \intn \intfaces{ \Up [r_h, \bm{u}_h] \jump{ \phi_h}} \dt
\br
& = -\sum_{i=1}^3 \int_0^{t^{n+1}} E_i(r_h,\phi)\dt + \int_{\tau}^{t^{n+1}} \intTd{r_h \vuh \cdot \Grad \phi} \dt,
\\ 
E_{s,\Grad\vt}(\phi) &= -\int_0^{t^{n+1}} \intfaces{\jump{ \phi_h} \avs{\frac{1}{\vth}}} + \int_0^{\tau} \intfaces{ \avs{\phi_h} \frac{\jump{\vth }}{\vthout \vth} } \dt
\br
&\quad
 - \left( \intTauTd{ \frac{ \kappa \phi}{ \vthout \vth} \abs{\Gradd \vth }^2 }
- \intTauTd{ \frac{\kappa}{\vth} \Gradd \vth \cdot \Grad \phi } \right),
\\
E_{s,res}(\phi) & = \intn \intTd{R_{s}(\phi_h) }
\end{align*}
with $R_{s}$ given in \eqref{entropy_dissipation} and $E_i(r_h,\phi)$ defined by
\begin{equation}\label{E_nf}
\begin{aligned}
& E_1(r_h, \phi) = \frac12 \intfaces{ |\avs{\vuh} \cdot \vc{n}| \jump{r_h } \jump{ \phi_h} } , \quad
E_2(r_h, \phi) = \frac14 \intfaces{ \jump{\vuh} \cdot \vc{n} \jump{r_h } \jump{ \phi_h} } ,
\\
& E_3(r_h, \phi)
 = - h^{\alpha+1} \intTd{ r_h \Delta_h \phi_h } ,
\quad E_4(r_h, \phi) = \intTd{ r_h \vuh \cdot \Big(\Grad \phi - \Gradh \phi_h \Big) } .
\end{aligned}
\end{equation}
\end{Remark}

\subsection{Proof of Lemma \ref{lem_C1}}\label{app-cf-m-en}
\begin{proof}[Proof of Lemma \ref{lem_C1}]
Choosing the test function in the consistency formulations \eqref{cP2-tor}--\eqref{cP3-tor} from the class
\begin{align}\label{tse0}
\bfphi \in C^2_c([0,T] \times \Of; \R^d), \quad \phi \in C^2_c([0,T] \times \Of)
\end{align}
we directly obtain the consistency formulation of density \eqref{cP1_D}. Moreover, the
consistency errors of the formulations \eqref{cP2_D}--\eqref{cP3_D} read
\begin{align*}
&e_{\vm}(\bfphi) = \widetilde{e_{\vm}}(\bfphi) + E_{\vm,\penl}(\bfphi), \quad E_{\vm,\penl}(\bfphi) = - \frac{1}{\penl} \intn \int_{\Osh \setminus \Os}{ \vuh \cdot \bfphi } \dxdt,
\\
&e_{s}(\phi) = \widetilde{e_{s}}(\phi) + E_{s,\penl}(\phi),
 \quad E_{s,\penl}(\phi) = - \frac{1}{\penl} \intn \int_{\Osh \setminus \Os} {(\vth -\vthB)\frac{\phi}{\vth}} \dxdt.
\end{align*}
Due to the choice of the test functions stated in \eqref{tse0}, we know that $\phi= \mathcal{O}(h)$ and $|\bfphi|= \mathcal{O}(h)$ on ${\Osh \setminus \Os} $. Thanks to \eqref{lm_0} we finish the proof by estimating
\begin{equation*}
\abs{E_{\vm,\penl}(\bfphi)} \aleq h^{3/2}\penl^{-1/2}, \quad \abs{E_{s,\penl}(\phi)} \aleq h^{3/2} \penl^{-1/2}.
\end{equation*}
\end{proof}

\subsection{Proof of Lemma \ref{lem_C4}}
\begin{proof}[Proof of Lemma \ref{lem_C4}]
Let $\tau \in [t^n, t^{n+1})$.
Realizing that $\left[ \intTd{r_h}\right]_0^{\tau} = \intn \intTd{D_t r_h} \dt$, cf. \cite[(2.17)]{FLS_IEE},
we obtain from the ballistic energy balance \eqref{BE0} with $\phi_h=\hvth:=\Pim \hvt$ that
\begin{align*}
 &\left[ \intTd{ \left(\frac{1}{2} \vrh |\vuh |^2 + c_v \vrh \vth - \vrh s_h \hvth \right) (t, \cdot)} \right]_{t=0}^{t=\tau}
 +\frac{1}{\penl} \intnOshB{ |\vuh|^2 + \frac{(\vth-\vthB)^2 }{\vth} }
 \br
 & \quad + \intn\intTdB{ \frac{ \kappa\avs{\hvth}}{\vth \vthout } \; \abs{\Gradd \vth}^2 + \frac{\hvth}{\vth }\difuh} \dt
 { + \intn D_s(\hvth) \dt}
\br
& \leq
- \intn \intTdB{ \vrh s_h D_t \hvth + \vrh s_h \vuh \cdot \Gradh \hvth - \kappa \avs{ \frac1{\vth}} \; \Gradd \vth \cdot \Gradd \hvth} \dt
\\& \quad + \intn R_B(\hvth)\, \dt - \intn R_s(\hvth)\, \dt.
\end{align*}
Recalling the estimates \eqref{ds} and \eqref{rs} we obtain 
\begin{align}\label{ee-rs}
\intn \abs{R_s(\hvth)} \dt \leq C_B' h^{\alpha+1} + \frac12 \intn D_s(\hvth) \dt,
\end{align}
where $c'_B$ depends on $\Un{\vr}, \Ov{\vr}, \Un{\vt}, \Ov{\vt}$ and $\norm{\hvt}_{L^\infty(0,T;W^{1,\infty}(\tor))}$.
In view of the following estimates
\begin{align*}
& \frac{1}{\penl} \int_{\tau}^{t^{n+1}}\intOshB{ |\vuh|^2 + \frac{(\vth-\vthB)^2 }{\vth} } \dt\geq 0,
\quad
\int_{\tau}^{t^{n+1}}\intTd{ \frac{ \kappa\avs{\hvth}}{\vth \vthout } \; \abs{\Gradd \vth}^2} \dt \geq 0,
\br
& \intn\intTd{ \frac{\hvth}{\vth } \difuh } \dt \geq \int_{0}^{\tau}\intTd{ \frac{\hvth}{\vth } \difuh} \dt = \intTauTd{ \frac{\hvt}{\vth }\difuh},
\\
& \intTd{ \left( \vrh s_h \hvth \right) (t, \cdot)} = \intTd{ \left( \vrh s_h \hvt \right) (t, \cdot)}
\hspace{1cm} \mbox{for any} ~ t \in[0,T],
\end{align*}
we have
\begin{align*}
 &\left[ \intTd{ \left(\frac{1}{2} \vrh |\vuh |^2 + c_v \vrh \vth - \vrh s_h \hvt \right) (t, \cdot)} \right]_{t=0}^{t=\tau}
 +\frac{1}{\penl} \intTauOshB{ |\vuh|^2 + \frac{(\vth-\vthB)^2 }{\vth} }
 \br
 & \quad + \intTauTdB{ \frac{ \kappa \hvt}{\vth \vthout } \; \abs{\Gradd \vth}^2 + \frac{\hvt}{\vth }\difuh}
\br &
\leq
- \intTauTdB{ \vrh s_h \partial_t \hvt + \vrh s_h \vuh \cdot \Grad \hvt - \frac{\kappa}{\vth} \; \Gradd \vth \cdot \Grad \hvt} + e_{B}(\hvt) + C'_ B h^{\alpha+1}
\end{align*}
with the consistency error
$e_{B}(\hvt) = E_{B, \Grad\vt} + E_{B, res} + \intn R_B(\hvth) \, \dt,$ 
and
\begin{align*}
&E_{B, \Grad\vt} = \kappa\intTauTd{ \frac{\hvt - \avs{\hvth} }{\vth \vthout } \; \abs{\Gradd \vth}^2 },
\\
E_{B, res} & = \intTauTdB{ \vrh s_h \partial_t \hvt + \vrh s_h \vuh \cdot \Grad \hvt - \kappa \frac1{\vth} \; \Gradd \vth \cdot \Grad \hvt}
\br
&- \intn \intTdB{ \vrh s_h D_t \hvth + \vrh s_h \vuh \cdot \Gradh \hvth - \kappa \avs{ \frac1{\vth}} \; \Gradd \vth \cdot \Gradd \hvth} \dt
\br
 & = \intTauTdB{ \vrh s_h ( \partial_t \hvt - D_t \hvth) + \vrh s_h \vuh \cdot (\Grad \hvt- \Gradh \hvth) }
\br
 &-\kappa \intTauTdB{ \avs{ \frac1{\vth} } \; \Gradd \vth \cdot (\Grad \hvt - \Gradd \hvth) + \left( \frac1{\vth} - \avs{\frac1{\vth}} \right) \Gradd \vth \cdot \Grad \hvt }
\br
 & - \int^{t^{n+1}}_{\tau} \intTdB{ \vrh s_h D_t \hvth + \vrh s_h \cdot \vuh \Gradh \hvth - \kappa \avs{ \frac1{\vth} } \; \Gradd \vth \cdot \Gradd \hvth} \dt .
\nonumber
\end{align*}

The rest is to estimate $e_B(\hvth)$. Firstly, let us look at the $\int_{0}^{\tau} \cdot \dt $-terms.
Due to $\hvt \in W^{2,\infty}((0,T) \times \tor)$ and assumption \eqref{HP} we have
\begin{align*}
& \Bigabs{\intTauTd{ \frac{ \kappa(\hvt - \avs{\hvth} )}{\vth \vthout } \; \abs{\Gradd \vth}^2 }} \aleq h \norm{\hvt}_{L^\infty(0,T; W^{1,\infty}(\tor))} \norm{\Gradd \vth}_{L^2((0,T)\times\tor)},
\\
& \Bigabs{\intTauTdB{ \vrh s_h ( \partial_t \hvt - D_t \hvth)+ \vrh s_h \vuh \cdot (\Grad \hvt- \Gradh \hvth) }} \aleq (\TS +h)\norm{\hvt}_{W^{2,\infty}((0,T) \times \tor)},
\\
& \Bigabs{\intTauTd{ \kappa \avs{ \frac1{\vth} } \; \Gradd \vth \cdot (\Grad \hvt - \Gradd \hvth)}} \aleq h \norm{\hvt}_{L^2(0,T;W^{2,2}(\tor))} \norm{\Gradd \vth}_{L^2((0,T)\times\tor)},
\\
& \Bigabs{\kappa \intTauTd{ \left( \frac1{\vth} - \avs{ \frac1{\vth} } \right) \Gradd \vth \cdot \Grad \hvt }} \aleq h \norm{\hvt}_{L^\infty(0,T; W^{1,\infty}(\tor))}\norm{\Gradd \vth}_{L^2((0,T)\times\tor)}^2.
\end{align*}
Secondly, we estimate $\int^{t^{n+1}}_{\tau} \cdot \dt $-terms as follows
\begin{align*}
& \Bigabs{\int^{t^{n+1}}_{\tau} \intTdB{ \vrh s_h D_t \hvth + \vrh s_h \vuh \cdot \Gradh \hvth} \dt } \aleq \TS \norm{\hvt}_{W^{1,\infty}((0,T) \times \tor)},
\\
& \Bigabs{\int^{t^{n+1}}_{\tau} \intTd{ \kappa \avs{ \frac1{\vth}} \; \Gradd \vth \cdot \Gradd \hvth} \dt } \aleq \TS^{1/2} \norm{\Gradd \vth}_{L^2((0,T)\times\tor)} \norm{\hvt}_{L^\infty(0,T; W^{1,\infty}(\tor))} .
\end{align*}
Finally, let us estimate the term $\intn R_B \, \dt $.
Recalling \cite[Appendix C]{BLMSY} we have
\begin{align*}
& \Bigabs{\intn R_{B,2}(\hvth) \dt} = \Bigabs{\intn E_1(\vrh s_h, \hvt) + E_2(\vrh s_h, \hvt)\dt} \aleq (h+h^{1-\alpha}) \norm{\hvt}_{L^\infty(0,T; W^{2,\infty}(\tor))}.
\end{align*}
Realizing that $\hvt|_{\Os} = \vtB$ and $\hvth|_{\OOh} = \vthB$ and thanks to \eqref{lm_0} we obtain
\begin{align*}
&\abs{(\hvth-\vthB)}(t,x)\aleq h \left( \norm{\hvt}_{L^\infty(0,T; W^{1,\infty}(\tor))} + \norm{\vtB}_{L^\infty(0,T; W^{1,\infty}(\tor))}\right) \quad \mbox{for any } \ (t,x)\in (0,T)\times \Och,
\br
&\Bigabs{\frac{1}{\penl} \intn\intOsh{ \frac{(\vth-\vthB) (\hvth-\vthB) }{\vth} } \dt} = \Bigabs{\frac{1}{\penl} \intn\int_{\Osh \setminus \OOh} { \frac{(\vth-\vthB) (\hvth-\vthB) }{\vth} } \dxdt }
\br & \aleq h^{3/2}\penl^{-1/2} .
\end{align*}
Applying \eqref{ap2} we obtain
\begin{align*}
\Bigabs{ \TS \intn\intTd{D_t (\vrh s_h) \cdot D_t \hvth} \dt} \aleq \TS \left( \norm{D_t \vrh }_{L^2((0,T)\times\tor)} +\norm{D_t \vth }_{L^2((0,T)\times\tor)}\right) \aleq (\TS)^{1/2}.
\end{align*}
Collecting above estimates we obtain \eqref{ec4} and finish the proof.
\end{proof}

\section{Proof of the error estimates}\label{app-ee}
 In this section we analyze the error estimates between the numerical approximations $(\vrh,\vuh,\vth)$ and the strong solution $(\tvr,\tvu,\tvt)$. 
In what follows we will firstly prove the error estimates for $\tau \in \{t^{k,-}\}_{k=0}^{N_T}$ and further for any $\tau\in[0,T]$.

\subsection{Relative energy inequality}
Taking suitable test functions in the consistency formulation
\begin{equation*}
\eqref{cP4}|_{\hvt=\tvt}
+ \eqref{cP1-tor}|_{\phi = \frac{\abs{\tvu}^2}{2} - \frac{\pd \Hc(\tvr,\tvt)}{\pd \vr} }
-\eqref{cP2-tor}|_{\phi = \tvu}
\end{equation*}
we obtain the following relative energy inequality
\begin{equation}\label{REI}
\begin{aligned}
&\left[ \RE{\vrh,\vuh,\vth}{\tvr,\tvu,\tvt} \right]_0^{\tau}
 +\frac{1}{\penl} \intTauOsh{ \left(|\vuh|^2 + \frac{(\vth-\vthB)^2 }{\vth} \right)}
\\& 
+\intTauTd{ \frac{\tvt}{ \vth} \left( 2\mu \left|\Dhuh - \bD(\tvu) \right|^2 + \lambda \left| \Divh \vuh - \Div \tvu \right|^2 \right) }
+\intTauTd{\frac{\kappa\tvt}{\vth \vthout }|\Gradd\vth-\Grad\tvt|^2}
\\& \leq
R_C +\sum_{i=1}^{9} R_i + R_{\penl}- R'_{\bS}- R'_{\vt} - R_{\bS} - R_{\vt} ,
\end{aligned}
\end{equation}
where 
\begin{align*}
&R_C = e_B(\tvt, \tau)+C_B' h^{\alpha +1} + e_{\vr}\left( \frac12\abs{\tvu}^2 - \frac{\pd \Hc(\tvr,\tvt)}{\pd \vr} , \tau \right) - \widetilde{e_{\vm}} (\tvu, \tau) ,
\\
& R_1=
 - \intTauTd{ \vrh (s_h - \ts) (\vuh-\tvu) \cdot \Grad \tvt } ,
 \quad
 R_2= - \intTauTd{ \vrh (\vuh - \tvu) \otimes (\vuh - \tvu): \Grad \tvu } ,
\\&
R_3 = \intTauTd{ \frac{\vrh( \tvu - \vuh )}{\tvr} \Big( \tvr (\pd_t \tvu + \tvu \cdot \Grad \tvu) + \Grad p(\tvr) - \Div\tbS \mathrm{1}_{\Of} \Big)},
\\&
R_4= \intTauTd{ \frac{ \vrh-\tvr}{\tvr} ( \tvu - \vuh )\cdot \Div \tbS \mathrm{1}_{\Of}} = \int_0^\tau \intOf{ \frac{ \vrh-\tvr}{\tvr} ( \tvu - \vuh )\cdot \Div \tbS } \dt,
\\&
R_5= \intTauTd{ ( \tp-p_h - \pd_\vr \tp (\tvr - \vrh) - \pd_\vt \tp (\tvt - \vth) )\Div \tvu } ,
\\&
R_6=
-\intTauTd{ \bigg( (\vrh-\tvr) (s_h- \ts) + \tvr \big(s_h- \ts - \pd_\vr \ts (\vrh -\tvr) - \pd_\vt \ts (\vth -\tvt) \big) \bigg)(\pd_t\tvt + \tvu \cdot \Grad \tvt ) } ,
\\&
R_7 = -\intTauTd{ (\vth -\tvt) \left( \frac{c_v \tvr}{ \tvt} (\partial_t \tvt + \tvu \cdot \Grad \tvt) + \tvr \Div \tvu - \frac{1}{\tvt} \left( \tbS : \Grad \tvu + \frac{\kappa \abs{\Grad \tvt}^2}{\tvt} + \mathrm{1}_{\Of} \tvt \Div\left( \frac{\kappa \Grad \tvt}{\tvt} \right)\right) \right)},
\\&
R_8 = \intTauTdB{ \mathrm{1}_{\Of} \tvu \cdot \Div \tbS + \tbS: \Grad \tvu } , \quad
R_9 = \intTauTdB{ \frac{\kappa \abs{\Grad \tvt}^2}{\tvt} +\mathrm{1}_{\Of} \tvt \Div\left( \frac{\kappa \Grad \tvt}{\tvt} \right)} ,
\\&
R_{\penl} = \frac{1}{\penl} \intTauOsh{ \vuh \cdot \tvu } ,
\quad R'_{\bS} =
 \intTauTdB{ \vuh \cdot \Div \tbS \;\mathrm{1}_{\Of} + \Gradh \vuh : \tbS }, 
\\&
 R'_{\vt} = \kappa \intTauTdB{\vth \mathrm{1}_{\Of}\Div\left(\frac{\Grad \tvt}{\tvt}\right) +\Gradd\vth \cdot \frac{\Grad \tvt }{\tvt} } ,
 \\&
R_{\bS} = \intTauTd{ \frac{\tvt}{ \vth} \left( 2\mu \left|\Dhuh - \frac{\vth}{ \tvt} \bD(\tvu) \right|^2 + \lambda \left| \Divh \vuh - \frac{\vth}{ \tvt}\Div \tvu \right|^2 \right) }
\\& \hspace{1cm}- \intTauTd{ \frac{\tvt}{ \vth} \left( 2\mu \left|\Dhuh - \bD(\tvu) \right|^2 + \lambda \left| \Divh \vuh - \Div \tvu \right|^2 \right) },
 \\&
 R_{\vt} = \kappa \intTauTdB{ \frac{\vth^2\vthout - \tvt^3}{\tvt^2\vth\vthout } \abs{\Grad \tvt}^2 - \Gradd\vth \cdot \frac{\Grad \tvt }{\tvt} + \frac{2\tvt -\vthout}{\vth \vthout} \Grad \tvt \cdot \Gradd \vth} .
\end{align*}

Further, applying the definition of $(\tvr,\tvu,\tvt)$, i.e.
\begin{align*}
& \pdt \tvr=0, \quad \pdt \tvt = 0 \quad \text{and} \quad \tvu = {\bf 0}\quad \mbox{ on } \Os, \quad
 \tvr (\pd_t \tvu + \tvu \cdot \Grad \tvu) + \Grad \tp = \Div \tbS \quad \mbox{ on } \Of,
 \br
 &\partial_t \tvt + \tvu \cdot \Grad \tvt + \frac{\tvt}{c_v} \Div \tvu = \frac{1}{c_v \tvr} \left( \tbS : \Grad \tvu + \frac{\kappa \abs{\Grad \tvt}^2}{\tvt} + \tvt \Div\left( \frac{\kappa \Grad \tvt}{\tvt} \right)\right) \quad \mbox{ on } \Of,
\end{align*}
we obtain
\begin{align*}
&R_3 = - \intTauOs{ \frac{\vrh\vuh}{\tvr} \cdot \Grad \tp} ,\quad R_7 = \kappa\intTauOs{(\vth -\tvt) \frac{ \abs{\Grad \tvt}^2}{\tvt^2} } , \quad R_8=0,
\\
& R_9 - R'_{\vt} = \kappa\intTauTdB{ (\Grad\tvt - \Gradd\vth ) \cdot \frac{\Grad\tvt}{\tvt} + (\tvt - \vth) \mathrm{1}_{\Of} \Div\left( \frac{ \Grad \tvt}{\tvt} \right)}.
\end{align*}
Together with $\Div \tvu = \mbox{tr}(\bD (\tvu)), \Divh \vuh = \mbox{tr}(\Dhuh)$, and \eqref{HP} we reorganize the relative energy inequality \eqref{REI} as
\begin{align}\label{REI-1}
&\left[ \RE{\vrh,\vuh,\vth}{\tvr,\tvu,\tvt} \right]_0^{\tau}
 +\frac{1}{\penl} \intTauOsh{ \left(|\vuh|^2 + (\vth-\vthB)^2 \right)}
\br&
 \quad
+\intTauTd{ \left|\Dhuh - \bD(\tvu) \right|^2 }
+\intTauTd{|\Gradd\vth-\Grad\tvt|^2}
\br&
\aleq
R_C + R_C' + \sum_{i=1}^7 R_i + R_{\penl} - (R_{\bS} + R_{\vt}),
\end{align}
where $R_C'= R_9 - R'_{\vt} - R'_{\bS} $ is the compatibility error.

\subsection{Estimates on $R_C$}\label{app-cer}
Now we proceed to estimate the consistency error $R_C$. Hereafter we frequently write $L^p(0,T;L^q(\tor))$ as $L^pL^q$ for simplicity.

\begin{Lemma}\label{lmSP3}
Let $(\vrh, \vuh, \vth)$ be a numerical solution obtained by the FV scheme \eqref{VFV} with $(\TS,h,\penl) \in (0,1)^3$ and $ \alpha \in [0,1)$. 
Let assumption \eqref{HP} hold and $\delta>0$ be an arbitrary constant.
Then it holds that
\begin{align}\label{es-1}
& \abs{R_C} \aleq \TS+ h+ h^{1+\alpha} + h^{(1-\alpha)/2} + \penl^{-1}h^3
+ \delta\norm{ \Dhuh - \bD(\tvu)}_{L^2((0,\tau)\times\tor)}^2
+ \delta \frac{\norm{\vth-\vthB}_{L^2((0,\tau)\times\Osh)}^2}{\penl}
\br
& + \delta \norm{ \Gradd \vth - \Grad \tvt }_{L^2((0,\tau)\times \tor)}^2 + \intTau{ \RE{\vrh,\vuh,\vth}{\tvr,\tvu,\tvt}}
\end{align}
for $\tau \in \{t^{k,-}\}_{k=0}^{N_T}$.
\end{Lemma}

\begin{proof}
First, we remind that $\tau \in \{t^{k,-}\}_{k=0}^{N_T}$, which implies $\int_{\tau}^{t^{n+1}} \cdot \; \dt = 0$ with $\tau = t^{n+1,-}$.
Denoting $\phi_{\vr} := \frac12\abs{\tvu}^2 - \frac{\pd \Hc(\tvr,\tvt)}{\pd \vr}$,
in what follows we analyze $e_{\vr}(\phi_{\vr}),\, \widetilde{e_{\vm}}(\tvu) , \, e_B(\tvt)$ 
term by term.
Specifically,
\begin{itemize}
\item $E_t$: Recalling \cite[(2.16)]{FLS_IEE} we have
\begin{align*}
&E_t(r_h,\phi) = \sum_{i=1}^3 E_t^{(i)}(r_h,\phi),\quad E_t^{(1)} (r_h,\phi)= -\intTd{ r_h^0 \frac1{\Delta t} \int_0^{\Delta t} \left( \phi(0) - \phi(t) \right) \dt },\\
&
E_t^{(2)}(r_h,\phi) = - \intTd{ r_h^n \frac1{\Delta t} \int_{t^n}^{t^{n+1}} \left( \phi(t+\TS) - \phi(\tau) \right) \dt },
\br
& E_t^{(3)}(r_h,\phi) = -\intTauTd{r_h (\pd_t \phi(t) - D_t \phi(t+\Delta t))}.
\end{align*}
Hence, with our assumption \eqref{HP} and uniform bounds \eqref{ap1} we have
\begin{align*}
& \Bigabs{E_t^{(1)} (r_h,\phi)} + \Bigabs{E_t^{(2)} (r_h,\phi)} \aleq \TS \norm{\pd_t \phi}_{L^{2}L^{2}} \norm{r_h}_{L^{\infty}L^2} \aleq \TS,
\br
& \Bigabs{E_t^{(3)} (r_h,\phi)} \aleq \TS \norm{\pd_t^2 \phi}_{L^{2}L^{2}}\norm{r_h}_{L^{\infty}L^{2}} \aleq \TS
\end{align*}
for $r_h = \vrh, \phi = \phi_{\vr} $ and $r_h = \vm_h, \phi = \tvu$.

\item $E_F$: Recalling \cite[Appendix C]{BLMSY} we have
\begin{align*}
&
\abs{\intTau{E_2(\vrh, \phi_{\vr})}}
\aleq \norm{\pd_x \phi}_{L^{\infty}L^{\infty}} \norm{\Gradd \vuh}_{L^2L^2} h^{(3-\alpha)/2} \aleq h^{1-\alpha}, \br
&
\abs{\intTau{E_2(\vm_h, \tvu)}} \aleq \norm{\pd_x \phi}_{L^{\infty}L^{\infty}} h \norm{\Gradd \vuh}_{L^2L^2} \aleq h^{(1-\alpha)/2}.
\end{align*}
Analogously to \cite[Proof of Theorem 5.5]{LSY_penalty}, applying \eqref{EXTE} and \eqref{lm_ms} we have
\begin{align*}
\abs{\intTau{ E_1(r_h, \phi ) } }
&\aleq \abs{h \sumi \intTauTd{ r_h \cdot \left( \pdmeshi (\pdedgei \Pim \phi)  \;  \Pim \abs{\avs{\uih} } + (\Pim \pdedgei{\Pim\phi}) \pdmeshi \abs{\avs{\uih} } \right)} } \br
&\aleq \sumi \intTauOch{\abs{r_h} \cdot \Pim \abs{\avs{\uih} } } + h \norm{\phi}_{L^{\infty}W^{1,\infty}} \norm{r_h}_{L^\infty L^2} \norm{\Gradh \vuh}_{L^2L^2} \br
&\aleq h+ \intTau{ \RE{\vrh,\vuh,\vth}{\tvr,\tvu,\tvt} } ,
\br
\abs{\intTau{ E_3(r_h, \phi ) } }
&\aleq h^{\alpha+1} + h^{\alpha} \intTauOch{\abs{r_h}} \aleq h^\alpha \left( h+ \intTau{ \RE{\vrh,\vuh,\vth}{\tvr,\tvu,\tvt} } \right),
\br
\abs{\intTau{ E_4(r_h, \phi ) } }
&\aleq h + \intTauOch{\abs{r_h\vuh}} \aleq h + \intTau{ \RE{\vrh,\vuh,\vth}{\tvr,\tvu,\tvt} }.
\end{align*}
for $r_h = \vrh, \phi = \phi_{\vr} $ and $r_h = \vm_h, \phi = \tvu$, where $\pdedgei$ is the $i$-th component of $\Gradd$. 
Hence, we obtain
\begin{align*}
&\abs{E_F(\vrh, \phi_{\vr})} + \abs{E_F(\vm_h,\tvu)} \aleq
h+ h^{(1-\alpha)/2}+h^{\alpha+1}+ 
\intTau{ \RE{\vrh,\vuh,\vth}{\tvr,\tvu,\tvt}}.
\end{align*}

\item $E_{\vm,\bS}$: Applying \eqref{lm_f2} and \eqref{lm_gradu_c} we obtain
\begin{align*}
\abs{E_{\vm,\bS}} & = \Abs{\intTauTd{ \bS_h : (\Grad \tvu - \Gradh \Pim \tvu)}}
 \aleq h + \intTauOch{ \abs{\bS_h}} \aleq h + \intTauOch{ \abs{\Dhuh}}
\br
& \aleq \frac{h}{\delta}+ \delta\norm{ \Dhuh - \bD(\tvu)}_{L^2((0,\tau)\times\tor)}^2 \aleq h + \delta\norm{ \Dhuh - \bD(\tvu)}_{L^2((0,\tau)\times\tor)}^2.
\end{align*}
Hereafter, we always omit the coefficient $1/\delta$ since $\delta$ is constant.
On the other hand, we keep writing  the coefficient $\delta$ in front of $\norm{ \Dhuh - \bD(\tvu)}_{L^2((0,\tau)\times\tor)}^2$ in order to control it with the left-hand-side of \eqref{REI-1} by choosing $\delta$ appropriately.

\item $E_{\vm,p}$: Applying \eqref{lm_f2} and \eqref{ap1} we have
\begin{align*}
\abs{E_{\vm,p}} & = \Abs{\intTauTd{ p_h  (\Div \tvu - \Divh \Pim \tvu)}} \aleq h + \intTauOch{ \abs{p_h}} \aleq h.
\end{align*}


\item $E_{B,\Grad \vt}$: Applying the interpolation error \eqref{proj} we directly obtain
\[
\abs{E_{B,\Grad \vt}} = \abs{\kappa\intTauTd{ \frac{\hvt - \avs{\hvth} }{\vth \vthout } \; \abs{\Gradd \vth}^2 }}\aleq h \norm{\Grad \tvt}_{L^{\infty}L^{\infty}} \aleq h.
\]

\item $E_{B, res}$: Applying \eqref{lm_f2}, \eqref{lm_u1_c}, \eqref{lm_gradvt_c} we have
\begin{align*}
&\abs{\intTauTdB{ \vrh s_h (\pd_t \tvt- D_t \Pim\tvt) }} \aleq \TS \norm{\pd_t^2\tvt}_{L^2L^2} \aleq \TS,
\br
&\abs{\intTauTdB{ \vrh s_h \vuh \cdot (\Grad \tvt- \Gradh \Pim\tvt) }} \aleq h + \intTauOch{ \abs{\vrh s_h \vuh}}
\br
& \hspace{2cm} \aleq h + \intTauOch{ \abs{\vuh}} \aleq h + \intTau{ \RE{\vrh,\vuh,\vth}{\tvr,\tvu,\tvt} },
\br
& \intTauTdB{ \avs{ \frac1{\vth}} \; \Gradd \vth \cdot (\Grad \tvt - \Gradd \Pim \tvt) } 
\br
& \hspace{2cm} \aleq h + \intTauOch{ \abs{\Gradd \vth}} \aleq h + \delta \norm{ \Gradh \vth - \Grad \tvt }_{L^2((0,\tau)\times \tor)}^2,
\br
&\abs{\intTauTd{ \left( \frac1{\vth} - \avs{ \frac1{\vth}} \right) \Gradd \vth \cdot \Grad \tvt } } \aleq h \norm{\tvt}_{L^\infty(0,T;W^{1,\infty}(\tor))}\norm{\Gradd \vth}_{L^2((0,T)\times\tor)}^2,
\end{align*}
which leads to
\begin{align*}
\abs{E_{B, res}} \aleq \TS + h + \delta \norm{ \Gradh \vth - \Grad \tvt }_{L^2((0,\tau)\times \tor)}^2 + \intTau{ \RE{\vrh,\vuh,\vth}{\tvr,\tvu,\tvt} }.
\end{align*}

\item $\intTau{ R_B(\Pim \tvt)}$: Thanks to $\Pim \tvt|_{\OOh} = \vthB|_{\OOh}, \ \abs{\Pim \tvt - \vthB}|_{\Osh \setminus \OOh} \aleq h $, applying \eqref{lm_vt0} with $\delta = \delta\frac{\penl}{h^2}$ we have
\begin{align*}
& \frac{1}{\penl} \abs{\intTau{\int_{\Osh \setminus \OOh} { \frac{(\vth-\vthB) (\Pim \tvt -\vthB) }{\vth} } \dx} }
\\ & \aleq \frac{h}{\penl} \intTau{\int_{\Osh \setminus \OOh} { \abs{\vth-\vthB} } \dx}
 \aleq \delta \frac{\norm{\vth-\vthB}_{L^2((0,\tau)\times\Osh)}^2}{\penl} + \penl^{-1}h^3.
\end{align*}
Applying \eqref{ap2} and
$D_t (\vrh s_h) \cdot D_t \tvt = D_t(\vrh s_h D_t \tvt) - (\vrh s_h)^{\triangleleft} D_t^2 \tvt
$
we have
\begin{align*}
&\abs{ \TS \intTauTd{D_t (\vrh s_h) \cdot D_t \Pim \tvt} }
=\abs{ \TS \intTauTd{D_t (\vrh s_h) \cdot D_t \tvt} }
\\&
\aleq \TS \abs{\intTauTd{D_t(\vrh s_h D_t \tvt)} } + \TS \abs{\intTauTd{(\vrh s_h)^{\triangleleft} D_t^2 \tvt} }
\br
&\aleq \TS \big( \norm{\pdt \tvt}_{L^\infty L^2} + \norm{\pd_t^2\tvt}_{L^2 L^2} \big) \aleq \TS.
\end{align*}
Applying the same approach for $\sum_{i=1}^2 \intTau{E_i(\vrh , \phi_\vr)}$ we obtain
\begin{align*}
\sum_{i=1}^2 \abs{\intTau{E_i(\vrh s_h, \tvt)}} \aleq h^{1-\alpha} + h + \intTau{ \RE{\vrh,\vuh,\vth}{\tvr,\tvu,\tvt}}.
\end{align*}
Collecting the above estimates we have
\begin{align*}
\abs{\intTau{ R_B(\Pim \tvt)}} \aleq & \TS + h+ h^{1-\alpha} + \penl^{-1}h^3+ \delta \frac{\norm{\vth-\vthB}_{L^2((0,\tau)\times\Osh)}^2}{\penl} \br
&+ \intTau{ \RE{\vrh,\vuh,\vth}{\tvr,\tvu,\tvt}} .
\end{align*}
\end{itemize}

Summing up the above estimates, together with $\alpha \in [0,1)$ we finish the proof with
\begin{align*}
& |R_C| \leq \Abs{e_{\vr}(\phi_{\vr})} + \Abs{e_{\vm} (\tvu)} + \Abs{e_B(\tvt)} + \Abs{C_B' h^{\alpha+1}}\br
&\aleq \TS +h^{1+\alpha} + h^{(1-\alpha)/2} + \penl^{-1}h^3
+ \delta\norm{ \Dhuh - \bD(\tvu)}_{L^2((0,\tau)\times\tor)}^2
+ \delta \frac{\norm{\vth-\vthB}_{L^2((0,\tau)\times\Osh)}^2}{\penl}
\br
&
+ \delta \norm{ \Gradd \vth - \Grad \tvt }_{L^2((0,\tau)\times \tor)}^2
+ (1+h^\alpha)\intTau{ \RE{\vrh,\vuh,\vth}{\tvr,\tvu,\tvt}}
.
\end{align*}
\end{proof}

\subsection{Estimates on $R_C'$}\label{app-cer-1}
Now let us estimate the compatibility error $R_C' = R'_{\bS} + R_9 - R'_{\vt}$, cf. \eqref{REI-1}, with
\begin{align*}
&R'_{\bS} = \intTauTdB{ \vuh \cdot \Div \tbS \mathrm{1}_{\Of} + \Gradh \vuh : \tbS },
\br
&R_9 - R'_{\vt} = \intTauTdB{ (\Grad\tvt - \Gradd\vth ) \frac{\Grad\tvt}{\tvt} + (\tvt - \vth) \mathrm{1}_{\Of} \Div\left( \frac{ \Grad \tvt}{\tvt} \right)}.
\end{align*}

\begin{Lemma}\label{lmSP4}
Let $(\vrh, \vuh, \vth)$ be a numerical solution obtained by the FV scheme \eqref{VFV} with $(\TS,h,\penl) \in (0,1)^3$ and $-1 < \alpha <1$. 
Let assumption \eqref{HP} hold and let $\delta>0$ be an arbitrary constant.
Then it holds
\begin{align}\label{es-2}
\abs{R_C'} \aleq& h^{(1-\alpha)/2} + h + \frac{\penl}{h} + \delta \frac{\norm{\vuh}_{L^2((0,\tau)\times\Osh)}^2+\norm{\vth-\vthB}_{L^2((0,\tau)\times\Osh)}^2}{\penl}
\br
&+ \delta \left( \norm{ \Dhuh - \bD(\tvu)}_{L^2((0,\tau) \times\tor)}^2 + \norm{ \Grad\tvt - \Gradd\vth}_{L^2((0,\tau) \times\tor)}^2 \right).
\end{align}
\end{Lemma}

\begin{proof}
{\bf The term $R'_{\bS}$.} Let us rewrite $\intTauTd{ \vuh \cdot \mathrm{1}_{\Of} \Div \tbS }$:
\begin{align*}
&\intTauTd{ \vuh \cdot \mathrm{1}_{\Of} \Div \tbS } = \intTauOfh{ \vuh \cdot \Div \tbS} + \int_0^{\tau} \int_{\Of \setminus \Ofh} \vuh \cdot \Div \tbS \dxdt
\br
&= \intTauOfh{ \vuh \cdot \Divmesh \Piw \tbS} + \int_0^{\tau} \int_{\Of \setminus \Ofh} \vuh \cdot \Div \tbS \dxdt
= \intTauTd{ \vuh \cdot \Divmesh \Piw \tbS} + I_{\vu,1}
\end{align*}
with
\begin{align*}
 I_{\vu,1} = - \int_0^{\tau} \int_{\Osh \setminus \OOh} \vuh \cdot \Divmesh \Piw \tbS \dxdt
 + \int_0^{\tau} \int_{\Of \setminus \Ofh} \vuh \cdot \Div \tbS \dxdt.
\end{align*}
Applying \eqref{EXTE} and \eqref{lm} we control $I_{\vu,1}$ with
\begin{align*}
\abs{I_{\vu,1}} \aleq (1+h^{-1}) \norm{\vu}_{L^{\infty}(0,T;W^{2,\infty}(\Of))} \int_0^{\tau} \int_{\Osh \setminus \OOh} \abs{\vuh} \dxdt \aleq h+ \frac{\penl}{h} + \delta \frac{\norm{\vuh}_{L^2((0,\tau)\times\Osh)}^2}{\penl}.
\end{align*}
Recalling the result in \cite[Appendix E]{Basaric} we have
\begin{align} \label{Iu2}
 & I_{\vu,2}: = \intTauTdB{ \vuh \cdot \Divmesh \Piw\tbS + \Gradh \vuh : \tbS }
 = \intTauTdB{ \vuh \cdot \Divmesh \Piw \tbS + \Gradh \vuh : \Pim \tbS }
\br&
 = -\intTau{\intfacesB{ \jump{\vuh} \cdot \left(\Piw \tbS - \avs{\Pim \tbS} \right)\cdot \vn}}
=-\intTauTd{ \Gradd \vuh : ( \Piw\tbS - \avs{\Pim \tbS} )}.
\end{align}
Applying \eqref{lm_gradu_c-1} we control $I_{\vu,2}$ with
\begin{equation}
\begin{aligned}
\abs{I_{\vu,2}} & \aleq h \norm{\Gradd \vuh}_{L^1L^1} + \intTauOch{|\Gradd \vuh|}
\br &
\leq h \norm{\Gradd \vuh}_{L^1L^1}+ 2 \intTau{\int_{\Och}{|\Gradh \vuh|} \dx} + 2 h^{-1} \intTau{\int_{\Omega_h^{C,1} }{|\vuh| \dx}}
\br &
\aleq h^{(1-\alpha)/2}+ h+ \delta \norm{ \Dhuh - \bD(\tvu)}_{L^2((0,\tau) \times\tor)}^2 + \frac{\penl}{h} + \delta \frac{\norm{\vuh}_{L^2((0,\tau)\times\Osh)}^2}{\penl},
\end{aligned}
\end{equation}

 where $\Omega_h^{C,1}$ is obtained by shifting $\Och$ one cell towards $\Os$ so that $\Omega_h^{C,1} \subset \Osh$.
Finally, we obtain
\[
\abs{R'_{\bS}} = \abs{I_{\vu,1} + I_{\vu,2}} \aleq h^{(1-\alpha)/2} + \frac{\penl}{h} + \delta \frac{\norm{\vuh}_{L^2((0,\tau)\times\Osh)}^2}{\penl} + \delta \norm{ \Dhuh - \bD(\tvu)}_{L^2((0,\tau) \times\tor)}^2.
\]

{\bf The term $(R_9 - R'_{\vt})$.}
Analogously as above, we reformulate temperature terms in the following way 
\begin{align*}
&
\intTauOf{ (\tvt - \vth) \Div\left( \frac{ \Grad \tvt}{\tvt} \right)}
= - \intTauTd{\Gradd (\Pim \tvt - \vth) \cdot \Piw \left( \frac{ \Grad \tvt}{\tvt} \right)} + \sum_{i=1}^4 I_{\vt,i},
\\ &
\intTauTd{ (\Grad\tvt - \Gradedge\vth ) \cdot\frac{\Grad\tvt}{\tvt} }
=
\intTauTd{ \Gradedge (\Pim\tvt - \vth) \cdot \Piw \left( \frac{ \Grad \tvt}{\tvt} \right)}+
I_{\vt,5} +I_{\vt,6} ,
\end{align*}
where
\begin{align*}
& I_{\vt,1} = \int_0^{\tau} \int_{\Of \setminus \Ofh} (\tvt - \vth) \Div\left( \frac{ \Grad \tvt}{\tvt} \right) \dxdt, \quad
I_{\vt,2} = \intTauOfh{( \tvt - \Pim \tvt) \Div\left( \frac{ \Grad \tvt}{\tvt} \right)},
\\&
I_{\vt,3} = \intTauOOh{ (\vth - \Pim \tvt ) \ \Divmesh \left(\Piw \frac{ \Grad \tvt}{\tvt} \right)},
 \\&
I_{\vt,4} =\intTau{\int_{\Osh \setminus \OOh}{ (\vth - \Pim \tvt ) \ \Divmesh \left(\Piw \frac{ \Grad \tvt}{\tvt} \right)}\dx},
\\&
I_{\vt,5} = \intTauTd{ (\Grad\tvt - \Gradedge\vth ) \cdot \left( \frac{\Grad\tvt}{\tvt} - \Piw \left( \frac{ \Grad \tvt}{\tvt} \right) \right)},
\\&
I_{\vt,6}= \intTauTd{ (\Grad\tvt - \Gradedge (\Pim\tvt )) \cdot \Piw \left( \frac{ \Grad \tvt}{\tvt} \right)}.
\end{align*}
Thanks to interpolation error estimates \eqref{proj} and estimate \eqref{lm_vt0} -- \eqref{lm_0} we control $I_{\vt,i}$ terms as follows
\begin{align*}
 \abs{I_{\vt,1}}
 & \aleq \norm{\vt}_{L^\infty(0,T;W^{2,\infty}(\Of))} \int_0^{\tau} \int_{\Of \setminus \Ofh} 1 \dxdt \aleq h,
\\
\abs{I_{\vt,2}}
&\aleq h \norm{\vt}_{L^\infty(0,T;W^{1,\infty}(\Of))} \norm{\vt}_{L^1(0,T;W^{2,1}(\Of))} \aleq h,
\\
 \abs{I_{\vt,3}} 
&\aleq 
\intTauOOh{ \abs{\vth - \vthB } }
\aleq
\delta \frac{\norm{\vth-\vthB}_{L^2((0,\tau)\times\Osh)}^2}{\penl}+ \frac{\penl}{\delta},
\\
 \abs{I_{\vt,4}} 
&\aleq
h^{-1} \intTau{\int_{\Osh \setminus \OOh}{ \abs{\vth - \Pim \tvt }}\dx}
\leq h^{-1} \intTau{\int_{\Osh \setminus \OOh}{ \left( \abs{\vth - \vthB }+ \abs{\vthB - \Pim \tvt } \right) }\dx }
\\& \aleq
 \delta \frac{\norm{\vth-\vthB}_{L^2((0,\tau)\times\Osh)}^2}{\penl} + \frac{\penl}{h \delta} +h,
\\
\abs{I_{\vt,5}} 
&\aleq h + \intTauOch{ \abs{\Grad\tvt - \Gradedge\vth}}
\aleq h + \delta \norm{ \Grad\tvt - \Gradedge\vth}_{L^2((0,\tau)\times\tor)}^2,
\\
\abs{I_{\vt,6}} 
& \aleq h+
\intTauOch{ \abs{\Piw \left( \frac{ \Grad \tvt}{\tvt} \right)}}
\aleq h.
\end{align*}
Consequently, we have
\begin{align*}
 \abs{R_9 - R'_{\vt}} = \abs{\sum_{i=1}^6 I_{\vt,i}
} \aleq h+ \frac{\penl}{h} + \delta \frac{\norm{\vth-\vthB}_{L^2((0,\tau)\times\Osh)}^2}{\penl} + \delta \frac{\norm{\vth-\vthB}_{L^2((0,\tau)\times\Osh)}^2}{\penl},
\end{align*}
which completes the proof.

\end{proof}


\subsection{Relative energy inequality revisited}
Let us continue to estimates the right-hand-side of \eqref{REI-1}.
Under assumption \eqref{HP}, we obtain from \eqref{EN}, \eqref{lm_u1} and \eqref{lm_vt1} that
\begin{equation}\label{es-3}
\begin{aligned}
&\Abs{R_1 + R_2 + R_4 + R_5 + R_6} \aleq \intTau{ \RE{\vrh,\vuh,\vth}{\tvr,\tvu,\tvt}},
\\
&\Abs{R_3} \aleq \intTauOs{ \Abs{ \vuh }} \aleq \delta \frac{\norm{\vuh}_{L^2((0,\tau )\times\Osh)}^2}{\penl} + \frac{\penl}{\delta},
\\
&\Abs{R_7} \aleq \intTauOs{ \Abs{ \vth - \tvt }} = \intTauOs{ \Abs{ \vth - \vtB }} \leq \intTauOsB{ \Abs{ \vth - \vthB } + \Abs{ \vthB - \vtB }}
\\
&\hspace{0.8cm}\aleq \delta \frac{\norm{\vth-\vthB}_{L^2((0,\tau )\times\Osh)}^2}{\penl} + \frac{\penl}{\delta} + h.
\end{aligned}
\end{equation}
Applying \eqref{EXTE1} and \eqref{lm_u0} with $\delta := \delta\frac{\penl}{h^2}$ we have
\begin{align*}
\Abs{R_{\penl}} \aleq \frac{h}{\penl} \intTau{ \int_{\Osh \setminus \Os}{ \abs{\vuh} } \dx } \aleq \delta \frac{\norm{\vuh}_{L^2((0,\tau)\times\Osh)}^2}{\penl} + \penl^{-1}h^3.
\end{align*}
Moreover, with $\Div = \mbox{tr}(\bD), \Divh = \mbox{tr}(\bD_h)$ and \eqref{S6} we have
\begin{align}\label{es-4}
\abs{R_{\bS}} \aleq & \abs{\intTauTd{ \frac{2\mu\tvt}{ \vth} \left( \left|\Dhuh - \frac{\vth}{ \tvt} \bD(\tvu) \right|^2 - \left|\Dhuh - \bD(\tvu) \right|^2 \right) } }
\br
=&\abs{\intTauTd{ \frac{2\mu\tvt}{ \vth} \left( \frac{\tvt - \vth}{\tvt} \bD(\tvu) : (\Dhuh - \bD(\tvu)) + \frac{(\tvt - \vth)^2}{\tvt^2} |\bD(\tvu)|^2\right)}}
\br
\aleq & \delta \norm{ \Dhuh - \bD(\tvu)}_{L^2((0,\tau )\times\tor)}^2 + \norm{\tvt - \vth}_{L^2((0,\tau )\times\tor)}^2
\br
\aleq & \delta \norm{ \Dhuh - \bD(\tvu)}_{L^2((0,\tau )\times\tor)}^2 + \intTau{ \RE{\vrh,\vuh,\vth}{\tvr,\tvu,\tvt} }.
\end{align}

All together, collecting \eqref{es-1}, \eqref{es-2}, \eqref{es-3}, and \eqref{es-4} and recalling \cite[Appendix E]{BLMSY}
\begin{align*}
 \Abs{R_{\vt}} \aleq h + \intTau{ \RE{\vrh,\vuh,\vth}{\tvr,\tvu,\tvt} }
+\delta \norm{\Gradd\vth-\Grad\tvt}_{L^2((0,\tau ) \times\tor)}^2
\end{align*}
we obtain the energy inequality via choosing $\delta \in (0,1)$:
\begin{equation}\label{RE04}
\begin{aligned}
&\left[ \RE{\vrh,\vuh,\vth}{\tvr,\tvu,\tvt} \right]_0^{\tau}
 +\frac{1}{\penl} \intTauOsh{\left( |\vuh|^2 + (\vth-\vthB)^2 \right)}\dt
\\& \hspace{2cm}+\intTauTd{\left|\Dhuh - \bD(\tvu) \right|^2 } + \intTauTd{|\Gradd\vth-\Grad\tvt|^2}
\\& \aleq \TS + h^{1+\alpha} + h^{(1-\alpha)/2} + \penl^{-1} h^3 + \penl+ \penl h^{-1} + \intTau{ \RE{\vrh,\vuh,\vth}{\tvr,\tvu,\tvt}}.
\end{aligned}
\end{equation}
Further, applying Gronwall's lemma, together with the error estimate caused by regular initial data
\begin{align*}
&\RE{\vrh,\vuh,\vth}{\tvr,\tvu,\tvt}(0) \aleq \intTdB{ \abs{\vuh^0 -\tvu(0)}^2 + \abs{ \vth^0 - \tvt(0)}^2 + \abs{\vrh^0 - \tvr(0)}^2 }
 \aleq h^2
\end{align*}
we obtain for any $\alpha \in [0,1)$ and for any $\tau \in \{t^{k,-}\}_{k=0}^{N_T}$ that
\begin{equation*}
\begin{aligned}
&\RE{\vrh,\vuh,\vth}{\tvr,\tvu,\tvt}(\tau)
 +\frac{1}{\penl} \intTauOsh{\left( |\vuh|^2 + (\vth-\vthB)^2 \right) }
\\&\hspace{2cm} +\intTauTd{\left|\Dhuh - \bD(\tvu) \right|^2 } + \intTauTd{|\Gradd\vth-\Grad\tvt|^2}
\\& \aleq \TS + h^{1+\alpha} + h^{(1-\alpha)/2} + \penl^{-1} h^3 +\penl h^{-1}.
\end{aligned}
\end{equation*}
For $\tau\in[0,T]$, thanks to
\begin{align*}
 & \norm{ (\vrh,\vu_h,\vth) - (\tvr,\tvu,\tvt) }_{L^\infty(0,\tau; L^2(\tor))} = \sup_{t \in [0,\tau] } \norm{ (\vrh,\vu_h,\vth) - (\tvr,\tvu,\tvt) }_{L^2(\tor)}
 \br
 & \hspace{2cm} \leq \sup_{t \in \{t^{k,-}\}_{k=0}^{N_T} } \norm{ (\vrh,\vu_h,\vth) - (\tvr,\tvu,\tvt) }_{L^2(\tor)} + \TS \norm{ \pd_t (\tvr,\tvu,\tvt) }_{L^\infty(0,T;L^2(\tor))}
 \br
 & \norm{ \Dhuh - \bD(\tvu) }_{L^2((0,\tau)\times \tor)} \leq \norm{\Dhuh - \bD(\tvu) }_{L^2((0,T)\times \tor)} ,
 \br
 & \norm{\Gradd \vth- \Grad \tvt}_{L^2((0,\tau)\times \tor)} \leq \norm{\Gradd \vth- \Grad \tvt}_{L^2((0,T)\times \tor)}
\end{align*}
we finish the proof.
\end{document}